\documentclass[12pt]{amsart}
\usepackage{amssymb,bm}
\usepackage{graphics}
\usepackage{graphicx}
\usepackage{color}
\newtheorem{theorem}{Theorem}[section]
\newtheorem{lemma}[theorem]{Lemma}
\newtheorem{corollary}[theorem]{Corollary}
\theoremstyle{remark}
\newtheorem{remark}[theorem]{Remark}
\begin{document}
\def\Date{July 30, 2014}  
\allowdisplaybreaks[4]
\numberwithin{equation}{section} \numberwithin{figure}{section}
\numberwithin{table}{section}
\def\LTS{{L_2(\sigma)}}
\def\O{\Omega}
\def\cT{\mathcal{T}}
\def\cB{\mathcal{B}}
\def\cE{\mathcal{E}}
\def\cF{\mathcal{F}}
\def\cI{\mathcal{I}}
\def\p{\partial}
\def\R{\mathbb{R}}
\def\bu{\bm{u}}
\def\bg{\bm{g}}
\def\bv{\bm{v}}
\def\bn{\bm{n}}
\def\bK{{\bm A}}
\def\bA{\bm{A}}
\def\bzeta{\bm{\zeta}}
\def\bbeta{\bm{\beta}}
\def\bxi{\bm{\xi}}
\def\bw{\bm{w}}
\def\bb{\bm{b}}
\def\div{\nabla\cdot}
\def\HDiv{{H(\text{div};\O)}}
\def\LT{{L_2(\O)}}
\def\ES{\HDiv\times\LT}
\def\ESA{[\LT]^d\times H^1_0(\O)}
\def\ESh{V_h\times Q_h}
\def\ESk{V_k\times Q_k}
\def\kES{V_k\times Q_k}
\def\PihRT{\Pi_h^{\rm RT}}
\def\hsa{h_{\scriptscriptstyle\sigma}}
\def\jump#1{[\hspace{-1.5pt}[#1]\hspace{-1.5pt}]_\sigma}
\def\mean#1{\{\hspace{-3pt}\{#1\}\hspace{-3pt}\}_\sigma}
\def\tW{\tilde W}
\def\tQ{\tilde Q}
\def\tq{\tilde q}
\def\tbl{(\!(\!(}
\def\tbr{)\!)\!)}
\def\tbb#1#2{\tbl #1,#2\tbr}
\def\tbar{|\!|\!|}
\def\Sh{\mathfrak{S}_h}
\def\fSk{\mathfrak{S}_k}
\def\Sk{\mathbb{S}_k}
\def\SSum{\sum_{\sigma\in\Sh}}
\def\SSumk{\sum_{\sigma\in\mathfrak{S}_k}}
\def\SiSumk{\sum_{\sigma\in\mathfrak{S}_k^i}}
\def\tbns#1{\tbar#1\tbar_{s,k}}
\def\dhNorm#1{\|#1\|_{d_h}}
\def\dkNorm#1{\|#1\|_{H^1(\O;\cT_k)}}
\def\bns{\bm{n}_\sigma}
\def\Bh{\mathbb{B}_h}
\def\Bk{\mathbb{B}_k}
\def\PH{{H^1(\O;\cT_h)}}
\def\PHk{{H^1(\O;\cT_k)}}
\def\PHkm{{H^1(\O;\cT_{k-1})}}
\def\PLT{{L_2(\O;\cT_h)}}
\def\PLTk{{L_2(\O;\cT_k)}}
\def\LTT{{L_2(T)}}
\def\PiT{\Xi_{T,\ell-1}}
\def\sss{\scriptscriptstyle}
\def\TSum{\sum_{T\in\cT_h}}
\def\TkSum{\sum_{T\in\cT_k}}
\def\db#1#2{(\!(#1,#2)\!)_k}
\def\hT{h_{\sss T}}
\def\FE#1{{V_{#1}\times Q_{#1}}}
\def\vq{(\bv,q)}
\def\wr{(\bw,r)}
\def\up{(\bu,p)}
\def\uph{(\bu_h,p_h)}
\def\gz{{(\bg,z)}}
\def\BSB{{\Bk\Sk\Bk}}
\def\SBB{{\Sk\Bk^2}}
\def\CGC{Id_k-I_{k-1}^kP_k^{k-1}}
\def\CGCT{Id_k-I_{k-1}^k\tP_k^{k-1}}
\def\tR{{\tilde R}}
\def\tS{{\tilde S}}
\def\tP{{\tilde P}}
\def\tE{{\tilde E}}
\title[Multigrid Methods for Darcy Systems]
{Multigrid Methods for Saddle Point Problems:\\
  Darcy Systems}
\author[S.C. Brenner]{Susanne C. Brenner}
\address{Susanne C. Brenner, Department of Mathematics and Center for
Computation and Technology, Louisiana State University, Baton Rouge,
LA 70803} \email{brenner@math.lsu.edu}
\author[D.-S. Oh]{Duk-Soon Oh} \address{Duk-Soon Oh,
 Department of Mathematics and Center for Computation and Technology,
 Louisiana State University, Baton Rouge, LA 70803}
\email{duksoon@cct.lsu.edu}
\author[L.-Y. Sung]{Li-yeng Sung} \address{Li-yeng Sung,
 Department of Mathematics and Center for Computation and Technology,
 Louisiana State University, Baton Rouge, LA 70803}
\email{sung@math.lsu.edu}
\begin{abstract}
 We design and analyze multigrid methods for the saddle point problems resulting from
 Raviart-Thomas-N\'ed\'elec mixed finite element methods (of order at least 1) for the Darcy system in porous
 media flow.  Uniform convergence of the $W$-cycle algorithm in a nonstandard energy norm is established.
 Extensions to general second order elliptic problems are also addressed.
\end{abstract}
\subjclass{65N55, 65N30, 65N15}
\thanks{The work of the first and third authors was supported in part
 by the National Science Foundation under Grant No.
 DMS-13-19172.}
\keywords{multigrid, saddle point problem, Darcy, second order elliptic problems,
 Raviart-Thomas-N\'ed\'elec mixed finite element methods}
\maketitle
\section{Introduction}\label{sec:Introduction}
 Multigrid methods for saddle point problems arising from mixed finite element methods for
 Stokes and Lam\'e systems were investigated in the recent paper \cite{BLS:2014:StokesLame}, where
 uniform convergence for the $W$-cycle algorithm in the energy norm was established for
 arbitrary polyhedral domains.
 In this paper we will extend the results in \cite{BLS:2014:StokesLame}
 to the Darcy system in porous media flow, and to general second order elliptic problems.
 We will follow the standard notation for differential
 operators and function spaces
 that can be found, for example, in \cite{Ciarlet:1978:FEM,BScott:2008:FEM,BBF:2013:Mixed}.
\par
 Let $\O$ be a polyhedral domain in $\R^d$ ($d=2,3$) occupied by a porous media.
 The velocity $\bu$ and pressure $p$ of a flow in $\O$ that obeys Darcy's law are determined by the
 system of equations
\begin{alignat}{3}
  \bu&=-\bK\nabla p&\qquad&\text{in $\O$},\label{eq:SDarcy1}\\
  \div\bu&=f&\qquad&\text{in $\O$},\label{eq:SDarcy2}\\
\intertext{together with the boundary condition}
  p&=g&\qquad&\text{on $\p\O$}\label{eq:SDarcy3}.
\end{alignat}
 Here $f$ is a source, $g$ is the pressure on $\p\O$, and
 $\bK$, a (sufficiently) smooth
 $d\times d$ symmetric positive definite (SPD) matrix function on $\bar\O$,
  is the permeability tensor divided by the viscosity.
\par
 For the design and analysis of multigrid methods, it suffices to consider the case where $g=0$.
 A standard weak formulation \cite{BBF:2013:Mixed} of \eqref{eq:SDarcy1}--\eqref{eq:SDarcy3} is then
 to find
 $(\bu,p)\in\HDiv\times\LT$ such that
\begin{alignat}{3}
  a(\bu,\bv)+b(\bv,p)&=0&\qquad&\forall\,\bv\in\HDiv,\label{eq:Darcy1}\\
  b(\bu,q)\hspace{48pt}    &=F(q)&\qquad&\forall\,q\in\LT,\label{eq:Darcy2}
\end{alignat}
 where
\begin{equation*}
   a(\bw,\bv)=\int_\O (\bK^{-1}\bw)\cdot\bv\,dx, \quad
   b(\bv,q)=-\int_\O (\div\bv)q\,dx \quad \text{and} \quad  F(q)=-\int_\O fq\,dx.
\end{equation*}
 Note that \eqref{eq:Darcy1}--\eqref{eq:Darcy2} can be written concisely as
\begin{equation}\label{eq:ConciseDarcy}
  \cB\big(\up,\vq)=F(q)\qquad\forall\,\vq\in\ES,
\end{equation}
 where
\begin{equation}\label{eq:BDef}
  \cB\big(\wr,\vq\big)=a(\bw,\bv)+b(\bv,r)+b(\bw,q).
\end{equation}
\par
 Let $\cT_h$ be a simplicial triangulation of $\O$.
  The Raviart-Thomas-N\'ed\'elec finite element method
  \cite{RT:1977:MFM,Nedelec:1980:EdgeElements} for \eqref{eq:Darcy1}--\eqref{eq:Darcy2} is to find
 $(\bu_h,p_h)\in \ESh$ such that
\begin{equation}\label{eq:RT}
  \cB\big(\uph,\vq\big)=F(q)\qquad\forall\,\bv\in \ESh,
\end{equation}
 where $V_h\subset \HDiv$ is the Raviart-Thomas-N\'ed\'elec vector finite element space of order $\ell\geq 1$
 associated with $\cT_h$ and
 $Q_h\subset \LT$ is the space of discontinuous piecewise $P_\ell$ functions.
\par
 We will consider, as in \cite{BLS:2014:StokesLame}, all-at-once multigrid methods that
 compute $\bu$ and $p$ simultaneously.  There is, however, a fundamental difference between the
 saddle point problems for the Stokes and Lam\'e systems considered in \cite{BLS:2014:StokesLame} and
 the saddle point problem \eqref{eq:RT}.
\par
 For the saddle point problems in
 \cite{BLS:2014:StokesLame}, the vector variable belongs to
 $[H^1(\O)]^d$ and the scalar variable belongs to $\LT$, which are
 the correct spaces for the duality argument that appears in the
 proof of the
 approximation properties of the multigrid algorithms.
 For the  saddle point problem \eqref{eq:RT}, the
 vector variable belongs to $\HDiv$ and the scalar variable belongs
 to $\LT$, which are {\em not} the correct spaces for the duality
 argument that is based on elliptic regularity (cf.
 \eqref{eq:EllipticRegularityEst} below).
\par
 This difficulty regarding the saddle point problem defined by
 \eqref{eq:RT} can be remedied by
 treating it as a {\em nonconforming} method for the following alternative weak formulation of the
 Darcy system:
   Find $(\bu,p)\in [\LT]^d\times H^1_0(\O)$ such that
\begin{alignat}{3}
  a(\bu,\bv)+b'(\bv,p)&=0&\qquad&\forall\,\bv\in[\LT]^d,\label{eq:ADarcy1}\\
  b'(\bu,q)\hspace{48pt}    &=F(q)&\qquad&\forall\,q\in H^1_0(\O),\label{eq:ADarcy2}
\end{alignat}
 where
 $$b'(\bv,q)=\int_\O \bv\cdot\nabla q\,dx.$$
\par
 The weak formulation \eqref{eq:ADarcy1}--\eqref{eq:ADarcy2} is well-defined for
 $F\in H^{-s}(\O)$ for $0\leq s\leq 1$, and we have the following
 elliptic regularity estimate (cf. \cite{Grisvard:1985:EPN,Dauge:1988:EBV,NP:1994:EPD}):
\begin{equation}\label{eq:EllipticRegularityEst}
  \|\bu\|_{H^\alpha(\O)}+\|p\|_{H^{1+\alpha}(\O)}\leq
   C_{\O,\bK}\|F\|_{H^{-1+\alpha}(\O)},
\end{equation}
 where $\alpha\in (\frac12,1]$ is determined by $\O$ and $\bK$, and
 $\alpha=1$ if $\O$ is convex. 
\begin{remark}\label{rem:RHS}
 Note that \eqref{eq:RT} is well-defined for $F\in H^{-s}(\O)$ as long as
 $s<\frac12$ since
 $Q_h$ is a subspace of $H^s(\O)$ for any $s<\frac12$.  Moreover
 \eqref{eq:ConciseDarcy} remains valid for the solution $(\bu,p)$ of
 \eqref{eq:ADarcy1}--\eqref{eq:ADarcy2} and for any  $\vq\in \ESh$, provided we use the following
 interpretation of $b(\cdot,\cdot)$:
  $$b(\bv,q)=-\langle \div\bv,q\rangle_{H^{-s}(\O)\times H^{s}(\O)}$$
 where $\langle\cdot,\cdot\rangle_{H^{-s}(\O)\times H^{s}(\O)}$ is the canonical bilinear form on
 ${H^{-s}(\O)\times H^{s}(\O)}$.
\end{remark}
\par
 The weak formulation defined by
 \eqref{eq:ADarcy1}--\eqref{eq:ADarcy2} provides the correct setting
 for a duality argument based on \eqref{eq:EllipticRegularityEst}, which allows us to establish uniform
 convergence for $W$-cycle algorithms for \eqref{eq:RT} in a nonconforming
 energy norm related to \eqref{eq:ADarcy1}--\eqref{eq:ADarcy2}.  This is also the reason that
 we require the order of the Raviart-Thomas-N\'ed\'elec finite element method to be at least 1,
 since piecewise constant functions provide poor approximations of functions in $H^1_0(\O)$
 (cf. Remark~\ref{rem:Order}).
\par
  We note that multigrid algorithms for the lowest order
 Raviart-Thomas-N\'ed\'elec finite element method can be developed through its connection to
 the Crouzeix-Raviart nonconforming $P_1$ finite element method
 \cite{CR:1973:NCPOne,AB:1985:MNC,Brenner:1992:RT}.  There are also other multilevel
 iterative solvers for the Darcy system.
 We refer the readers to
 \cite{RW:1992:Mixed,EW:1994:MLMixed,AEL:1992:Darcy,RVW:1996:Mixed,BGL:2005:SaddlePoint,
 Vassilevski:2008:Book,MW:2011:SaddlePoint}
  for a discussion of such methods.
\par
 The rest of the paper is organized as follows.  We present the nonstandard error analysis for the
 Raviart-Thomas-N\'ed\'elec finite element method in Section~\ref{sec:Error}.  The results in this section
 are important for the convergence analysis of the multigrid methods and also shed new light
 on these finite element methods.  We introduce the multigrid algorithms in Section~\ref{sec:MG}
 and mesh-dependent norms in Section~\ref{sec:MeshDependentNorms}, which are important tools for
 the convergence analysis carried out in Section~\ref{sec:Analysis}.  In
 Section~\ref{sec:General} we extend the results to
 mixed finite element methods for generalized Darcy systems arising from general second order elliptic problems.
 Numerical results are presented in Section~\ref{sec:Numerics}, followed by
 some concluding remarks in Section~\ref{sec:Conclusion}.
\par
 Throughout the paper we will use $C$ (with or without subscripts) to denote a generic positive constant
 that depends only on the domain $\O$, the order $\ell$ of
 the finite element spaces and the shape regularity of the triangulations, but not the mesh sizes.
 To avoid the proliferation of constants, we also use the notation $A\lesssim B$ (or $A\gtrsim B$)
 to represent
 $A\leq\text{(generic constant)}\times B$.  The notation $A\approx B$ is equivalent to
 $A\lesssim B$ and $A\gtrsim B$.

\section{A Nonstandard Error Analysis for Raviart-Thomas-N\'ed\'elec \\ Finite Element Methods}\label{sec:Error}
 In this section
 we will carry out the error analysis of \eqref{eq:RT} as a
 nonconforming finite element method for \eqref{eq:ADarcy1}--\eqref{eq:ADarcy2}. The analysis is based on
 mesh-dependent norms and the saddle point theory of Babu\v ska \cite{Babushka:1973:LM}
  and Brezzi \cite{Brezzi:1974:SPP}.  Similar ideas have been applied to the analysis of
  mixed finite element methods for the biharmonic problem \cite{BOP:1980:Mixed}.
\subsection{Mesh-Dependent Norms for the Finite Element Spaces}\label{subsec:MDN}
 The norm $\|\cdot\|_{\PLT}$ on $[H^\alpha(\O)]^d + V_h$ is defined by
\begin{equation}\label{eq:PLTNorm}
  \|\bv\|_{\PLT}^2=\sum_{T\in\cT_h}\|\bv\|_\LTT^2+\SSum\hsa\|\bv\cdot\bns\|_\LTS^2,
\end{equation}
 where $\Sh$ is the set of the sides (faces for $d=3$ and edges for $d=2$) of the elements in $\cT_h$,
 $\hsa$ is the diameter of the side $\sigma$, and
 $\bns$ is a unit normal of $\sigma$.
\par
 Note that
\begin{equation}\label{eq:NormEquivalence}
  \|\bv\|_\LT\leq \|\bv\|_\PLT\lesssim \|\bv\|_\LT \qquad\forall\,\bv\in V_h.
\end{equation}
\par
 Let $\Pi_h$ be the nodal interpolation operator for the Raviart-Thomas-N\'ed\'elec finite element space
 $V_h$.  It is well-known \cite{Nedelec:1980:EdgeElements,BBF:2013:Mixed} that
\begin{equation}\label{eq:RTInterpolationEst1}
  \|\bzeta-\Pi_h\bzeta\|_\LT\lesssim h^{s}|\bzeta|_{H^s(\O)}\quad\text{for}\quad \frac12<s\leq \ell+1,
\end{equation}
 where $h=\max_{T\in\cT_h}\text{diam}\,T$ is the mesh size.
 We also have, by
 a standard argument based on the Bramble-Hilbert lemma \cite{BH:1970:Lemma,DS:1980:BH},
\begin{equation}\label{eq:RTInterpolationEst2}
  \SSum\hsa\|(\bzeta-\Pi_h\bzeta)\cdot\bns\|_\LTS^2\lesssim
  h^{2s}|\bzeta|_{H^s(\O)}^2\quad\text{for}\quad \frac12<s\leq \ell+1.
\end{equation}
\par
 The norm $\|\cdot\|_\PH$ on $H^1(\O)+Q_h$ is defined by
\begin{equation}\label{eq:PHNorm}
  \|q\|_{\PH}^2=\sum_{T\in\cT_h}\|\nabla q\|_\LTT^2+\SSum\frac{1}{\hsa}\|\jump{q}\|_\LTS^2,
\end{equation}
 where $\jump{q}$ is the jump of $q\in Q_h$ across a side $\sigma\in\Sh$ defined as follows.
\par
 If $\sigma$ is interior to $\O$, then $\sigma$ is the common side of the elements
 $T_\pm$ and
\begin{equation}\label{eq:InteriorJump}
  \jump{q}=q_-\bm{n}_{\sigma,-} + q_+\bm{n}_{\sigma,+},
\end{equation}
 where $q_\pm=q\big|_{T_\pm}$ and $\bm{n}_{\sigma,\pm}$ is the unit normal of $\sigma$ pointing
 towards the outside of $T_\pm$.
\par
 If $\sigma$ is on $\p\O$, then $\sigma$ is the side of a unique element $T_\sigma$ in $\cT_h$ and
\begin{equation}\label{eq:BoundaryJump}
  \jump{q}=q_{\sss T_\sigma}\bm{n}_{\sss T_\sigma},
\end{equation}
 where $q_{\sss T_\sigma}=q\big|_{T_\sigma}$ and $\bm{n}_{\sss T_\sigma}$ is the unit normal of $\sigma$ pointing
 towards the outside of $T_\sigma$.
\par
 The norm $\|\cdot\|_\PH$ is a well-known norm in the analysis of discontinuous Galerkin methods for second
 order problems \cite{ABCM:2001:DG,BScott:2008:FEM}, and we have a standard interpolation error estimate
\begin{equation}\label{eq:DGInterpolationEst}
  \|\phi-\cI_h\phi\|_\PH \lesssim h^{s}|\phi|_{H^{1+s}(\O)}\quad\text{for}\quad
   0<s\leq \ell,
\end{equation}
 where $\cI_h$ is the nodal interpolation operator for the conforming $P_\ell$ Lagrange finite element space.
\begin{remark}\label{rem:Order}
  The estimate \eqref{eq:DGInterpolationEst} implies
     $$\lim_{h\downarrow0}\inf_{q\in Q_h}\|p-q\|_\PH=0,$$
  which is not true if $\ell=0$.  This is the reason why we only consider Raviart-Thomas-N\'ed\'elec
  finite element methods of order $\ell\geq1$.
\end{remark}
\begin{remark}\label{rem:VW}
  The connection between the DG norm $\|\cdot\|_\PH$ and the Raviart-Thomas-N\'ed\'elec
   finite element method was exploited in \cite{RVW:1996:Mixed} for the preconditioning
   of the saddle point problem \eqref{eq:RT}.
\end{remark}
\subsection{Stability Estimates}\label{subsec:Stability}
 Since $\bK$ is a smooth symmetric positive definite matrix on $\bar\O$,
 we have the obvious estimates
\begin{alignat}{3}
   a(\bw,\bv)&\lesssim\|\bw\|_\LT\|\bv\|_\LT&\qquad&\forall\,\bv,\bw\in [\LT]^d,\label{eq:aBdd}\\
   a(\bv,\bv)&\gtrsim\|\bv\|_\LT^2\gtrsim \|\bv\|_\PLT^2&\qquad&\forall\,\bv\in V_h.\label{eq:aCoercive}
\end{alignat}
\par
 Let $\alpha$ be the index of elliptic regularity that appears in \eqref{eq:EllipticRegularityEst}.
 It follows from integration by parts and \eqref{eq:InteriorJump}--\eqref{eq:BoundaryJump} that
\begin{align}\label{eq:IBP}
  -\langle \div\bv,q\rangle_{H^{-1+\alpha}(\O)\times H^{1-\alpha}(\O)}
  &=\TSum\Big(-\int_{\p T} (\bn_T\cdot\bv)q\,ds+\int_T \bv\cdot\nabla q\,dx\Big)\\
     &=-\SSum\int_\sigma \bv\cdot\jump{q}\,ds+\sum_{T\in\cT_h}\int_T \bv\cdot\nabla q\,dx\notag
\end{align}
 for all $\bv\in [H^\alpha(\O)]^d+V_h$ and $q\in H^1(\O)+Q_h$, and therefore
\begin{align}\label{eq:bBdd}
  b(\bv,q) &\leq \Big(\SSum\hsa\|\bv\cdot\bns\|_\LTS^2\Big)^\frac12
           \Big(\SSum\hsa^{-1}\|\jump{q}\|_\LTS^2\Big)^\frac12\notag\\
       &\hspace{30pt}+
           \Big(\sum_{T\in\cT_h}\|\bv\|_{L_2(T)}^2\Big)^\frac12
           \Big(\sum_{T\in\cT_h}\|\nabla q\|_{L_2(T)}^2\Big)^\frac12\\
            &\leq \|\bv\|_\PLT\|q\|_\PH\notag
\end{align}
  for all $\bv\in [H^\alpha(\O)]^d+V_h$ and $q\in H^1(\O)+Q_h$.
\par
 Given any $q\in Q_h$  (a piecewise $P_\ell$ function), we define $\bv_q\in V_h$ by
\begin{alignat}{3}
  \bv_q\cdot\bns &=-\frac{1}{\hsa}\jump{q}\cdot\bns &\qquad&\forall\,\sigma\in\Sh,\label{eq:bv1}\\
  \PiT\bv_q&=\nabla q_{\sss T}&\qquad&\forall\,T\in \cT_h,\label{eq:bv2}
\end{alignat}
 where $\PiT$ is the orthogonal projection from $[\LTT]^d$  onto  $[P_{\ell-1}(T)]^d$.
 It follows from \eqref{eq:bv1}, \eqref{eq:bv2}, the definition of the Raviart-Thomas-N\'ed\'elec element
 \cite{Nedelec:1980:EdgeElements,BBF:2013:Mixed}
 and scaling that
\begin{align}\label{eq:bvNormEquivalence}
  \|\bv_q\|_\PLT^2&\approx \TSum \|\PiT\bv_q\|_\LTT^2+\SSum \hsa\|\bv_q\cdot\bns\|_\LTS^2\\
                &=\TSum \|\nabla q\|_\LTT^2+\SSum\frac{1}{\hsa}\|\jump{q}\|_\LTS^2=\|q\|_\PH^2.\notag
\end{align}
 On the other hand \eqref{eq:IBP}, \eqref{eq:bv1} and \eqref{eq:bv2} imply
\begin{equation}\label{eq:InfSupEquality}
  b(\bv_q,q)=\| q\|_\PH^2.
\end{equation}
 Combining \eqref{eq:bBdd},
 \eqref{eq:bvNormEquivalence} and \eqref{eq:InfSupEquality}, we arrive at the inf-sup
 condition
\begin{equation}\label{eq:InfSupCondition}
 \sup_{\bv\in V_h}\frac{b(\bv,q)}{\|\bv\|_\PLT}\approx \|q\|_\PH\qquad\forall\,q\in Q_h.
\end{equation}
\par
 It follows from \eqref{eq:aBdd}, \eqref{eq:aCoercive}, \eqref{eq:bBdd}, \eqref{eq:InfSupCondition}
 and the saddle point theory \cite{Babushka:1973:LM,Brezzi:1974:SPP,BBF:2013:Mixed} that
\begin{align}\label{eq:StabilityEstimate}
   &\|\bv\|_\PLT+\|q\|_\PH\\
   &\hspace{50pt}\approx
    \sup_{(\bw,r)\in\ESh}\frac{\cB\big((\bv,q),(\bw,r)\big)}{\|\bw\|_\PLT+\|r\|_\PH}
    \qquad\forall\, (\bv,q)\in\ESh.\notag
\end{align}
\par
 Note that \eqref{eq:aBdd}, \eqref{eq:bBdd} and \eqref{eq:BDef} imply
\begin{equation}\label{eq:BBdd}
  \cB((\bv,q),(\bw,r)\big)\lesssim (\|\bv\|_\PLT+\|q\|_\PH)(\|\bw\|_\PLT+\|r\|_\PH)
\end{equation}
 for all  $\bv,\bw\in [H^\alpha(\O)]^d+V_h$ and $q,r\in H^1(\O)+Q_h$,

\subsection{Error Estimates}\label{subsec:Error}
  Let $\alpha$ be the index of elliptic regularity in
 \eqref{eq:EllipticRegularityEst} and $F\in H^{-1+\alpha}(\O)$.
  According to Remark~\ref{rem:RHS}, the system \eqref{eq:RT} is well-defined and the solution
 $(\bu,p)\in [L_2(\O)]^d\times H^1_0(\O)$ of \eqref{eq:ADarcy1}--\eqref{eq:ADarcy2}
 satisfies
\begin{equation*}
   \cB\big((\bu,p),(\bv,q)\big)=F(q)\qquad\forall\,(\bv,q)\in\ESh.\\
\end{equation*}
 Consequently we have the Galerkin relation
\begin{equation}\label{eq:GalerkinRelation}
  \cB\big((\bu,p),(\bv,q)\big)=\cB\big((\bu_h,p_h),(\bv,q)\big) \qquad\forall\,(\bv,q)\in\ESh.
\end{equation}
\par
 Let $(\bv,q)\in \ESh$ be arbitrary.
 It follows from \eqref{eq:StabilityEstimate}--\eqref{eq:GalerkinRelation} that
\begin{align*}
  &\|\bv-\bu_h\|_\PLT+\|q-p_h\|_\PH\\
    &\hspace{50pt}\approx \sup_{(\bw,r)\in\ESh}\frac{\cB\big((\bv-\bu_h,q-p_h),(\bw,r)\big)}{\|\bw\|_\PLT+\|r\|_\PH}\\
    &\hspace{50pt}=\sup_{(\bw,r)\in\ESh}\frac{\cB\big((\bv-\bu,q-p),(\bw,r)\big)}{\|\bw\|_\PLT+\|r\|_\PH}
    \lesssim \|\bv-\bu\|_\PLT+\|q-p\|_\PH
\end{align*}
 and hence
\begin{align*}
  &\|\bu-\bu_h\|_\PLT+\|p-p_h\|_\PH \\
    &\hspace{40pt}\leq \|\bu-\bv\|_\PLT+\|p-q\|_\PH+
       \|\bv-\bu_h\|_\PLT+\|q-p_h\|_\PH\\
    &\hspace{40pt} \lesssim \|\bu-\bv\|_\PLT+\|p-q\|_\PH,
\end{align*}
 which then implies the quasi-optimal error estimate
\begin{equation}\label{eq:QuasiOptimalError}
  \|\bu-\bu_h\|_\PLT+\|p-p_h\|_\PH\lesssim \inf_{\bv\in V_h}\|\bu-\bv\|_\PLT+\inf_{q\in Q_h}\|p-q\|_\PH.
\end{equation}
\par
 Putting \eqref{eq:EllipticRegularityEst}, \eqref{eq:RTInterpolationEst1}, \eqref{eq:RTInterpolationEst2},
 \eqref{eq:DGInterpolationEst}
 and \eqref{eq:QuasiOptimalError} together, we have
\begin{equation}\label{eq:ConcreteError}
  \|\bu-\bu_h\|_\PLT+\|p-p_h\|_\PH\lesssim h^\alpha \|F\|_{H^{-1+\alpha}(\O)},
\end{equation}
 and in the case where $p\in H^m(\O)$ for $m\leq \ell+1$,
\begin{equation}\label{eq:SmoothError}
   \|\bu-\bu_h\|_\PLT+\|p-p_h\|_\PH\lesssim h^{m-1}|p|_{H^{m}(\O)}.
\end{equation}
\begin{remark}\label{rem:StandardErrors}
 The estimates \eqref{eq:ConcreteError} and \eqref{eq:SmoothError} in the nonconforming energy norm
  $\|\cdot\|_\PLT+\|\cdot\|_\PH$ are more informative than the
 standard error estimates for the Raviart-Thomas-N\'ed\'elec finite element methods in
 \cite{RT:1977:MFM,FO:1980:Mixed,BBF:2013:Mixed} since they provide approximations of the flux on the element interfaces.
 One can also recover the standard error estimates from \eqref{eq:ConcreteError}--\eqref{eq:SmoothError}.
\end{remark}
\section{Multigrid Methods}\label{sec:MG}
 We will introduce the multigrid methods for \eqref{eq:RT} in this section.
 The operators involved are defined with respect to a mesh-dependent inner product, and the
 smoothers for pre-smoothing and post-smoothing are defined in terms of a block-diagonal preconditioner.
\subsection{Set-Up}\label{subsec:SetUp}
 Let $\cT_0$ be an initial triangulation of $\O$ and the triangulations $\cT_1,\cT_2,\ldots$
 be obtained from $\cT_0$ through uniform subdivisions.  Since the Raviart-Thomas-N\'ed\'elec finite
 element pairs $V_k\times Q_k$ associated with $\cT_k$ are nested, we take the coarse-to-fine intergrid
 transfer operator
 $I_{k-1}^k:\FE{k-1}\longrightarrow \FE{k}$ to be the natural injection and define the Ritz projection
 operator $P_k^{k-1}:\FE{k}\longrightarrow\FE{k-1}$ by
\begin{equation}\label{eq:RitzProjection}
  \cB\big(P_k^{k-1}(\bv,q),(\bw,r)\big)=\cB\big((\bv,q),I_{k-1}^k(\bw,r)\big)
\end{equation}
 for all $(\bv,q)\in\FE{k}$ and $(\bw,r)\in\FE{k-1}$.
\par
 Let $(\cdot,\cdot)_k$ be a mesh-dependent inner product on $V_k$ such that
\begin{equation}\label{eq:VectorIP}
  (\bv,\bv)_k\approx \|\bv\|_\LT^2\qquad\forall\,\bv\in V_k
\end{equation}
 and the nodal basis (vector) functions for $V_k$ are orthogonal with respect to $(\cdot,\cdot)_k$.
 Similarly, let $\db{\cdot}{\cdot}$ be a mesh-dependent inner product on $Q_k$ such that
\begin{equation}\label{eq:ScalarIP}
  \db{q}{q}\approx \|q\|_\LT^2  \quad \forall\,q\in Q_k
\end{equation}
 and the nodal basis functions for $Q_k$ are orthogonal with respect to $\db{\cdot}{\cdot}$.
\begin{remark}\label{rem:DIPs}
  The inner products $(\cdot,\cdot)_k$ and $\db{\cdot}{\cdot}$ are constructed by
  mass lumping. 
\end{remark}
\par
 The mesh-dependent inner product $[\cdot,\cdot]_k$ on $\FE{k}$ is then defined by
\begin{equation}\label{eq:MDIP}
   \big[\vq,\wr\big]_k=h_k^2(\bv,\bw)_k
      +\db{q}{r},
\end{equation}
 where $h_k=\max_{T\in\cT_k}\text{diam}\, T$ is the mesh size of $\cT_k$.
 We take the fine-to-coarse intergrid transfer operator $I_k^{k-1}:\FE{k}\longrightarrow\FE{k-1}$
 to be the transpose of $I_{k-1}^k$ with respect to the mesh-dependent inner products on $\FE{k}$
 and $\FE{k-1}$, i.e.,
\begin{equation}\label{eq:FTC}
  \big[I_k^{k-1}\vq,\wr\big]_{k-1}=\big[\vq,I_{k-1}^k\wr\big]_k
\end{equation}
 for all $\vq\in \FE{k}$ and $\wr\in\FE{k-1}$.
\par
 Let the system operator $\Bk:\FE{k}\longrightarrow\FE{k}$ be defined by
\begin{equation}\label{eq:BkDef}
  [\Bk\vq,\wr]_k=\cB\big(\vq,\wr\big)\qquad\forall\,\vq,\wr\in\FE{k}.
\end{equation}
 Our goal is to develop multigrid algorithms for problems of the form
\begin{equation}\label{eq:GeneralProblem}
  \Bk\vq=\gz.
\end{equation}
%
\subsection{A Block-Diagonal Preconditioner}\label{subsec:Preconditioner}
 Let $L_k:Q_k\longrightarrow Q_k$ be an operator that is SPD with respect to
  $\db{\cdot}{\cdot}$ and  satisfies
\begin{equation}\label{eq:LkDef}
  \db{L_k^{-1}q}{q}\approx \|q\|_\PH^2 \qquad\forall\,q\in Q_k.
\end{equation}
 Then the preconditioner $\Sk:\FE{k}\longrightarrow \FE{k}$ given by
\begin{equation}\label{eq:bSkDef}
   \Sk(\bv,q)=(h_k^2\bv,L_k q)
\end{equation}
 is SPD with respect to $[\cdot,\cdot]_k$
 and we have
\begin{equation}\label{eq:SkInverse}
  [\Sk^{-1}(\bv,q),(\bv,q)]_k\approx \|\bv\|_\LT^2+\dkNorm{q}^2\qquad\forall\,(\bv,q)\in\kES
\end{equation}
 by \eqref{eq:VectorIP}--\eqref{eq:MDIP} and \eqref{eq:LkDef}.
\begin{remark}
  The operator $L_k$ can be constructed through  multigrid \cite{GK:2003:DG,BZhao:2005:MGDG2}
  or domain decomposition \cite{FK:2001:TLASDG,LT:2003:Overlapping,AA:2007:DDDG}
\end{remark}
\par
 The following result connects the operators $\Bk$, $\Sk$ and the nonconforming energy norm
 for $\FE{k}$.
\begin{lemma}\label{lem:BkSkBk}
  The norm equivalence
\begin{equation}\label{eq:BkSkBk}
  \Big[\BSB\vq,\vq\big]_k^\frac12 \approx \|\bv\|_\LT+\|q\|_\PH\qquad\forall\,\vq\in\FE{k}
\end{equation}
 holds for $k=0,1,2,\ldots$.
\end{lemma}
\begin{proof}  Let $\vq\in\FE{k}$ be arbitrary and $(\bm{x},y)=\Sk\Bk\vq$.  It follows from
 \eqref{eq:NormEquivalence}, \eqref{eq:StabilityEstimate}, \eqref{eq:BkDef}, \eqref{eq:SkInverse}
 and duality that
\begin{align*}
  \big[\BSB\vq,\vq\big]_k^\frac12&=\big[\Sk^{-1}(\Sk\Bk)\vq,(\Sk\Bk)\vq\big]_k^\frac12\\
     &=\big[\Sk^{-1}(\bm{x},y),(\bm{x},y)\big]_k^\frac12\\
     &=\sup_{\wr\in \FE{k}}\frac{\big[\Sk^{-1}(\bm{x},y),\wr\big]_k}
              {\big[\Sk^{-1}\wr,\wr\big]_k^\frac12}\\
     &\approx \sup_{\wr\in \FE{k}}\frac{\big[\Bk\vq,\wr\big]_k}
              {\|\bw\|_\LT+\|r\|_\PH}\\
     &=\sup_{\wr\in \FE{k}}\frac{\cB\big(\vq,\wr\big)}
              {\|\bw\|_\LT+\|r\|_\PH}
     \approx \|\bv\|_\LT+\|q\|_\PH.
\end{align*}
\end{proof}
\par
 Let $\rho(\BSB)$ be the spectral radius of the operator $\BSB$.
 It follows from \eqref{eq:MDIP}, \eqref{eq:BkSkBk} and a standard inverse estimate
 \cite{Ciarlet:1978:FEM,BScott:2008:FEM} that
\begin{equation}\label{eq:SpectralRadius}
  \rho(\BSB)\lesssim h_k^{-2}.
\end{equation}
 We can therefore choose a damping factor $\delta_k$ of the form
 $Ch_k^2$ such that
\begin{equation}\label{eq:DampingFactor}
  \delta_k\cdot\rho(\BSB)\leq 1.
\end{equation}
\subsection{Multigrid Algorithms}\label{subsec:NG}
 Let the output of the $W$-cycle algorithm for \eqref{eq:GeneralProblem} with initial
 guess $(\bv_0,q_0)$ and $m_1$ (resp. $m_2$) pre-smoothing (resp. post-smoothing) steps
 be denoted by $MG_W(k,\gz,(\bv_0,q_0),m_1,m_2)$.
\par
 We use a direct solve for $k=0$, i.e., we take
 $MG_W(0,\gz,(\bv_0,q_0),m_1,m_2)$ to be $\mathbb{B}_0^{-1}\gz$. For $k\geq1$, we compute
 $MG_W(k,\gz,(\bv_0,q_0),m_1,m_2)$ in three steps.
\par\medskip\noindent
{\em Pre-Smoothing} \quad The approximate solutions $(\bv_1,q_1),\ldots,(\bv_{m_1},q_{m_1})$
 are computed recursively  by
\begin{equation}\label{eq:PreSmoothing}
  (\bv_{j},q_{j})=(\bv_{j-1},q_{j-1})+\delta_k \Sk\Bk\big(\gz-\Bk(\bv_{j-1},q_{j-1})\big)
\end{equation}
 for $1\leq j\leq m_1$,
 where the damping factor $\delta_k$ satisfies \eqref{eq:DampingFactor}.
\par\medskip\noindent
{\em Coarse Grid Correction} \quad Let $(\bm{g}',z')=I_k^{k-1}\big(\gz-\Bk(\bv_{m_1},q_{m_1})\big)$ be the
 transferred residual of $(\bv_{m_1},q_{m_1})$ and compute $(\bv_1',q_1'),(\bv_2',q_2')\in \FE{k-1}$ by
\begin{align}
 (\bv_1',q_1')&=MG_W(k-1,(\bm{g}',z'),({\bf 0},0),m_1,m_2),\label{eq:CGC1}\\
 (\bv_2',q_2')&=MG_W(k-1,(\bm{g}',z'),(\bv_1',q_1'),m_1,m_2).\label{eq:CGC2}
\end{align}
 We then take $(\bv_{m_1+1},q_{m_1+1})$ to be $(\bv_{m_1},q_{m_1})+I_{k-1}^k(\bv_2',q_2')$.
\par\medskip\noindent
{\em Post-Smoothing} \quad The approximate solutions
 $(\bv_{m_1+1},q_{m_1+1}),\ldots,(\bv_{m_1+m_2+1},q_{m_1+m_2+1})$
 are computed recursively  by
\begin{equation}\label{eq:PostSmoothing}
  (\bv_{j},q_{j})=(\bv_{j-1},q_{j-1})+\delta_k \Bk\Sk\big(\gz-\Bk(\bv_{j-1},q_{j-1})
\end{equation}
 for $m_1+2\leq j\leq m_1+m_2+1$.
\par\smallskip
 The final output is $MG_W(k,\gz,(\bv_0,q_0),m_1,m_2)=(\bv_{m_1+m_2+1},q_{m_1+m_2+1})$.
\par\medskip
 Let $MG_V(k,\gz,(\bv_0,q_0),m_1,m_2)$ be the output of the $V$-cycle algorithm for
 \eqref{eq:GeneralProblem} with initial
 guess $(\bv_0,q_0)$ and $m_1$ (resp. $m_2$) pre-smoothing (resp. post-smoothing) steps.
 The computation of
 $MG_V(k,\gz,(\bv_0,q_0),m_1,m_2)$ differs from the computation for the $W$-cycle algorithm
 only in the coarse grid correction step, where we compute
  $$(\bv_1',q_1')=MG_V(k-1,(\bm{g}',z'),({\bf 0},0),m_1,m_2)$$
 and take $(\bv_{m_1+1},q_{m_1+1})$ to be $(\bv_{m_1},q_{m_1})+I_{k-1}^k(\bv_1',q_1')$.
%
\subsection{Error Propagation Operators}\label{subsec:ErrorPropagation}
 The effect of one post-smoothing step defined by \eqref{eq:PostSmoothing} is measured by
\begin{equation}\label{eq:RkDef}
  R_k=Id_k-\delta_k\BSB,
\end{equation}
 where $Id_k:\FE{k}\longrightarrow\FE{k}$ is the identity operator.  The choice of
 the smoother $\Bk\Sk$ for post-smoothing is motivated by the fact that \eqref{eq:RkDef}
 is the error propagation operator of one Richardson relaxation step for the SPD problem
\begin{equation}\label{eq:SPDProblem}
  \BSB\vq=\Bk\Sk\gz,
\end{equation}
 which is equivalent to \eqref{eq:GeneralProblem}.
\par
 On the other hand, the effect of one pre-smoothing step defined by \eqref{eq:PreSmoothing} is measured by
\begin{equation}\label{eq:SkDef}
  S_k=Id_k-\delta_k\SBB.
\end{equation}
 Our choice of the smoother $\Sk\Bk$ for the pre-smoothing is motivated by the adjoint relation
\begin{equation}\label{eq:AdjointRelation}
  \cB\big(R_k\vq,\wr\big)=\cB\big(\vq,S_k\wr\big) \qquad\forall\,\vq,\wr\in\FE{k}
\end{equation}
 that follows from \eqref{eq:BkDef}, \eqref{eq:RkDef} and \eqref{eq:SkDef}.
\par
 The error propagation operator $E_k:\FE{k}\longrightarrow\FE{k}$ for the multigrid algorithms
 satisfies the well-known recursive relation \cite{Hackbusch:1985:MMA,BZ:2000:AMG,BScott:2008:FEM}
\begin{equation}\label{eq:MGRecursion}
  E_k=R_k^{m_2}(Id_k-I_{k-1}^kP_k^{k-1}+I_{k-1}^k E_{k-1}^p P_k^{k-1})S_k^{m_1}\quad\text{for}\;k\geq 1,
\end{equation}
 where $P_k^{k-1}$ is the Ritz projection operator defined in \eqref{eq:RitzProjection}
 and $p=2$ (resp. $1$) for the $W$-cycle (resp. $V$-cycle) algorithm.
\par
 Since $I_{k-1}^k$ is the natural injection, we have
\begin{equation}\label{eq:Projection}
  P_k^{k-1}I_{k-1}^k=Id_{k-1}, \qquad (Id_k-I_{k-1}^kP_k^{k-1})^2=Id_k-I_{k-1}^kP_k^{k-1},
\end{equation}
 and the Galerkin orthogonality
\begin{equation}\label{eq:GalerkinOrthogonality}
  0=\cB\big((Id_k-I_{k-1}^kP_k^{k-1})\vq,I_{k-1}^k\wr\big)
\end{equation}
 that is valid for all $\vq\in\FE{k}$ and $\wr\in\FE{k-1}$.
%
\section{Mesh-Dependent Norms for Multigrid Analysis}\label{sec:MeshDependentNorms}
 We introduce in this section a scale of mesh-dependent norms that are crucial for
 the convergence analysis of the $W$-cycle multigrid algorithm in Section~\ref{sec:Analysis}.
\subsection{Definition of the Mesh-Dependent Norms}\label{subsec:MDNorms}
 For $0\leq s\leq 1$, we define the scale of mesh-dependent norms $\|\cdot\|_{s,k}$
 in terms of the SPD operator $\BSB$ and the mesh-dependent inner product $[\cdot,\cdot]_k$
 as follows:
\begin{equation}\label{eq:MDNorms}
  \|\vq\|_{s,k}^2=\big[(\BSB)^s\vq,\vq\big]_k \qquad\forall\,\vq\in\FE{k}.
\end{equation}
\par
 In view of \eqref{eq:VectorIP}--\eqref{eq:MDIP}, \eqref{eq:BkSkBk} and \eqref{eq:MDNorms},
 we have the obvious norm equivalences
\begin{alignat}{3}
  \|\vq\|_{0,k}^2&\approx h_k^2\|\bv\|_\LT^2+\|q\|_\LT^2&\qquad& \forall\,\vq\in\FE{k},
  \label{eq:ZeroEquivalence}\\
  \|\vq\|_{1,k}^2&\approx \|\bv\|_\LT^2+\|q\|_\PHk^2&\qquad& \forall\,\vq\in\FE{k}.
  \label{eq:OneEquivalence}
\end{alignat}
 Thus the $\|\cdot\|_{1,k}$ norm is equivalent to the nonconforming energy norm on $\FE{k}$
 and we have the following stability result.
\begin{lemma}\label{lem:Stability}
  The operators $I_{k-1}^k$ and $P_k^{k-1}$ are stable with respect to the mesh-dependent
  norm $\|\cdot\|_{1,k}$.
\end{lemma}
\begin{proof}
 Since $I_{k-1}^k$ is the natural injection,
 the stability estimate
\begin{align*}
   \|I_{k-1}^k\wr\|_{1,k}&\approx \|\bw\|_\LT+\|r\|_\PHk \\
    &\lesssim \|\bw\|_\LT+\|r\|_\PHkm
   \approx \|\wr\|_{1,k-1} \quad\forall\,\wr\in \FE{k-1}
\end{align*}
 follows from \eqref{eq:PHNorm}, \eqref{eq:OneEquivalence} and a direct calculation.
\par
 The stability of $P_k^{k-1}$ then follows from \eqref{eq:NormEquivalence},
 \eqref{eq:StabilityEstimate}, \eqref{eq:BBdd}, \eqref{eq:RitzProjection}, \eqref{eq:OneEquivalence}
 and duality:
\begin{align*}
  \|P_k^{k-1}\vq\|_{1,k-1}&\approx \sup_{\wr\in \FE{k-1}}
    \frac{\cB\big(P_k^{k-1}\vq,\wr\big)}{\|\wr\|_{1,k-1}}\\
    &=\sup_{\wr\in \FE{k-1}}
    \frac{\cB\big(\vq,I_{k-1}^k\wr\big)}{\|\wr\|_{1,k-1}}\\
    &\lesssim \|\vq\|_{1,k}  \hspace{140pt}\forall\,\vq\in\FE{k}.
\end{align*}
\end{proof}
\par
 We will need a connection between $\|\cdot\|_{1-\alpha,k}$ and a Sobolev norm in the
 proof of the approximation property in Section~\ref{sec:Analysis}.
 Towards this goal we introduce
 the operator $D_k:Q_k\longrightarrow Q_k$ defined by
\begin{equation}\label{eq:DkDef}
  \db{D_k q}{r}=\TkSum\int_T\nabla q\cdot\nabla r\,dx+\SSumk\hsa^{-1}\int_\sigma \jump{q}\cdot\jump{r}\,ds
  \qquad\forall(q,r)\in\FE{k}.
\end{equation}
 Then $D_k$ is SPD with respect to $\db{\cdot}{\cdot}$ and the relations
\begin{alignat}{3}
  \db{D_k^0q}{q}&\approx \|q\|_\LT^2&\qquad&\forall\,q\in Q_k, \label{eq:TrivialRelation1}\\
  \db{D_kq}{q}&\approx \|q\|_\PHk^2 &\qquad&\forall\,q\in Q_k, \label{eq:TrivialRelation2}
\end{alignat}
 follow immediately from \eqref{eq:PHNorm}, \eqref{eq:ScalarIP} and \eqref{eq:DkDef}.
\begin{remark}\label{rem:Lk}
  The operator $L_k$ that appears in \eqref{eq:bSkDef} is just an optimal preconditioner of $D_k$.
\end{remark}
\par
  It follows from standard inverse estimates that $\rho(D_k)\lesssim h_k^{-2}$
  and hence we have, by the spectral theorem,
\begin{equation}\label{eq:DGNormInverseEst}
  \|q\|_\PHk^2=\db{D_k q}{q}\leq Ch_k^{2(s-1)}\db{D_k^s q}{q}\qquad\forall\,q\in Q_k.
\end{equation}
\par
 In view of \eqref{eq:ZeroEquivalence}, \eqref{eq:OneEquivalence}, \eqref{eq:TrivialRelation1}
 and \eqref{eq:TrivialRelation2}, we have
\begin{alignat*}{3}
  \|(\bv,q)\|_{0,k}^2&\approx h_k^2\|\bv\|_\LT^2+\db{D_k^0 q}{q} &\qquad&\forall\,q\in Q_k,\\
  \|(\bv,q)\|_{1,k}^2&\approx \|\bv\|_\LT^2+\db{D_k^1 q}{q}&\qquad&\forall\,q\in Q_k,
\end{alignat*}
 which imply, through interpolation between Hilbert scales \cite[Chapter~23]{Tartar:2007:Sobolev},
 the norm equivalence
\begin{equation}\label{eq:IntermediateNormEquivalence}
  \|(\bv,q)\|_{s,k}^2\approx h_k^{2(1-s)}\|\bv\|_\LT^2+\db{D_k^s q}{q}\qquad\forall \, (\bv,q)\in\FE{k}
\end{equation}
 that holds for $0\leq s\leq 1$.
\par
 It only remains to relate $\db{D_k^s q}{q}$ to Sobolev norms, which will require certain tools from
 the multigrid theory for nonconforming finite element methods \cite{Brenner:1999:CNM,BZhao:2005:MGDG2}.
\subsection{Enriching and Forgetting Operators}\label{subsec:EnrichForget}
 Let $\tQ_k\subset H^1(\O)$ be the $P_{d+1+\ell}$ Lagrange finite element
 space associated with $\cT_k$. The {\em enriching} operator
 $\cE_k:Q_k\longrightarrow\tQ_k$ is defined by  averaging, i.e.,
\begin{equation}\label{eq:EkDef}
  (\cE_k q)(x)=\frac{1}{|\cT_x|}\sum_{T\in\cT_x} q_{\sss T}(x),
\end{equation}
 where $x$ is any node for $\tQ_k$, $\cT_x$ is the set of the elements in $\cT_k$ that share the node $x$,
 and $|\cT_x|$ is the number of elements in $\cT_x$.
\par
 The following estimate is obtained by a straight-forward local calculation:
\begin{equation}\label{eq:EnrichingEst}
  \sum_{T\in\cT_k}\hT^{-2} \|q-\cE_k q\|_\LTT^2\lesssim
        \SiSumk \frac{1}{\hsa}\|\jump{q}\|_\LTS^2\qquad\forall\,q\in Q_k,
\end{equation}
 where $\hT=\text{diam}\,T$ and $\mathfrak{S}_k^i$ is the set of the interior faces.
\par
 Since $q$ and $\cE_kq$ agree at the $(\ell+1)(\ell+2)/2$ interior nodes for each
 $T\in\cT_k$ when $d=2$ and the $(\ell+1)(\ell+2)(\ell+3)/6$
  interior nodes for each $T\in\cT_k$ when $d=3$, we can define a {\em forgetting} operator
    $\cF_k:\tQ_k\longrightarrow Q_k$ element by element so that
\begin{equation}\label{eq:RightInverse}
  \cF_k\circ\cE_k=Id_k
\end{equation}
 as follows.  For any $\tq\in\tQ_k$, we define $\cF_k\tq$ to be the (unique) function $q\in Q_k$
 such that, for any $T\in\cT_k$,
  $q=\tq$ at the nodes of $\tQ_k$ interior to $T$.  We have, by scaling,
\begin{equation}\label{eq:ForgettingEst}
\|\tq-\cF_k\tq\|_\LTT\leq C\hT|\tq|_{H^1(T)}\qquad\forall\,\tq\in\tQ_k,\;T\in\cT_k.
\end{equation}
\par
 The estimates \eqref{eq:EnrichingEst} and \eqref{eq:ForgettingEst} then imply, through standard inverse estimates
 \cite{Ciarlet:1978:FEM,BScott:2008:FEM}, 
\begin{alignat}{3}
   \|\cE_k q\|_{H^1(\O)}&\leq C\|q\|_\PHk&\qquad&\forall\, q\in Q_k,\label{eq:cEkHOneEst}\\
   \|\cE_kq\|_\LT&\leq C\|q\|_\LT&\qquad&\forall\,q\in Q_k,\label{eq:cEkLTwoEst}\\
   \|q-\cE_kq\|_\LT&\leq C_sh_k^s\|q\|_{H^s(\O)}&\qquad&\forall\,q\in Q_k,\,0\leq s<\frac12,
   \label{eq:cEkFractionalEst}\\
   \|\cF_k \tq\|_\PHk&\leq C\|\tq \|_{H^1(\O)}&\qquad&\forall\, \tq\in \tQ_k,\label{eq:cFkHOneEst}\\
   \|\cF_k\tq\|_\LT&\leq C\|\tq\|_\LT&\qquad&\forall\, \tq\in \tQ_k.\label{eq:cFkLTwoEst}
\end{alignat}
\subsection{Equivalence between Mesh-Dependent Norms and Sobolev Norms}
\label{subsec:DkChracterization}
 We will connect the mesh-dependent norms $\bm \|\cdot\|_{s,k}$
  to the Sobolev norms through two lemmas.
\begin{lemma}\label{lem:EnrichDGNormEquivalence}
  The norm equivalence
\begin{equation*}
    \db{D_k^s q}{q}\approx \|\cE_kq\|_{H^s(\O)}^2 \qquad \forall\, q\in Q_k
\end{equation*}
 holds for $0\leq s\leq 1$.
\end{lemma}
\begin{proof}
   It follows from the estimates \eqref{eq:TrivialRelation1}, \eqref{eq:TrivialRelation2},
    \eqref{eq:cEkHOneEst}, \eqref{eq:cEkLTwoEst} and interpolation
    between Hilbert scales that
\begin{equation*}
  \|\cE_k q\|_{H^s(\O)}\lesssim \db{D_k^s q}{q}^\frac12\qquad\forall\,q\in Q_k.
\end{equation*}
\par
 In order to prove the estimate in the opposite direction, we introduce the operator
 $$J_k=\cF_k\circ\Lambda_k,$$
 where
 $\Lambda_k:L_2(\O)\longrightarrow \tQ_k$ is the orthogonal projection.
 In view of \eqref{eq:RightInverse}, we have
\begin{equation}\label{eq:JkcEk}
  J_k\cE_k q =\cF_k\Lambda_k\cE_kq=\cF_k\cE_kq=q\qquad\forall\,q\in Q_k.
\end{equation}
 Moreover, it follows from \eqref{eq:TrivialRelation1}, \eqref{eq:TrivialRelation2},
 \eqref{eq:cFkHOneEst}, \eqref{eq:cFkLTwoEst}
 and the well-known estimate \cite{BX:1991:LTwo}
\begin{equation*}
  \|\Lambda_k\zeta\|_{H^1(\O)}\lesssim \|\zeta\|_{H^1(\O)}\qquad\forall\,\zeta\in H^1(\O)
\end{equation*}
 that
\begin{alignat*}{4}
   \db{D_k^0J_k \zeta}{J_k \zeta}^\frac12&\approx
    \|J_k \zeta\|_\LT&\lesssim\|\zeta\|_\LT&\qquad&\forall\,\zeta\in \LT,\\
   \db{D_k^1J_k \zeta}{J_k \zeta}^\frac12
   &\approx\|J_k\zeta\|_\PHk&\lesssim \|\zeta\|_{H^1(\O)}&\qquad&\forall\,\zeta\in H^1(\O).
\end{alignat*}
 The two last estimates imply, by interpolation between Hilbert scales,
\begin{equation*}
  \db{ D_k^s J_k\zeta}{J_k\zeta}^\frac12\leq C\|\zeta\|_{H^s(\O)}\qquad\forall\,\zeta\in H^s(\O),
\end{equation*}
 and hence, because of \eqref{eq:JkcEk},
\begin{equation*}
  \db{D_k^s q}{q}^\frac12=\db{D_k^s J_k\cE_kq}{J_k\cE_k q}^\frac12\lesssim
  \|\cE_kq\|_{H^s(\O)}\qquad\forall\,q\in Q_k.
\end{equation*}
\end{proof}
\begin{lemma} \label{lem:DGNormEquivalence}
  For any $s\in[0,\frac12)$, we have
\begin{equation*}
  \db{D_k^sq}{q}\approx |q|_{H^s(\O)}^2\qquad\forall\,q\in Q_k,
\end{equation*}
 where the constants in the norm equivalence depend on $s$.
\end{lemma}
\begin{proof}  Using the (non-standard) inverse estimate \cite{BB:2001:Nonstandard}
\begin{equation}\label{eq:NonStandardInverseEst}
   \|q-\cE_kq\|_{H^s(\O)}\leq C_s h_k^{-s}\|q-\cE_kq\|_\LT\qquad\forall\,q\in Q_k
   \quad\text{and}\quad 0\leq s<\frac12
\end{equation}
 together with \eqref{eq:DGNormInverseEst},
 \eqref{eq:EnrichingEst} and Lemma~\ref{lem:EnrichDGNormEquivalence},
 we find
\begin{align*}
   \|q\|_{H^s(\O)}&\leq \|q-\cE_kq\|_{H^s(\O)}+\|\cE_kq\|_{H^s(\O)}\\
                  &\lesssim h_k^{-s}\|q-\cE_kq\|_{L_2(\O)}+ \db{D_k^sq}{q}^\frac12\\
                  &\lesssim h_k^{1-s}\|q\|_\PHk+\db{D_k^sq}{q}^\frac12
                  \lesssim \db{D_k^sq}{q}^\frac12\qquad\forall\, q\in Q_k.
\end{align*}
\par
 In the other direction we have, by \eqref{eq:cEkFractionalEst},
 Lemma~\ref{lem:EnrichDGNormEquivalence} and \eqref{eq:NonStandardInverseEst},
\begin{align*}
  \db{D_k^sq}{q}^\frac12 &\lesssim \|\cE_kq\|_{H^s(\O)}\\
     &\lesssim \big(\|q\|_{H^s(\O)}+\|q-\cE_kq\|_{H^s(\O)}\big)\\
     &\lesssim \big(\|q\|_{H^s(\O)}+h_k^{-s}\|q-\cE_kq\|_\LT\big)\lesssim \|q\|_{H^s(\O)}\qquad\forall\,q\in Q_k.
\end{align*}
\end{proof}
\par
 Combining \eqref{eq:IntermediateNormEquivalence} and Lemma~\ref{lem:DGNormEquivalence}, we arrive at
 the following result.
\begin{corollary} \label{cor:Comparison}
 For any $s\in [0,\frac12)$, we have
\begin{equation*}
  \|(\bv,q)\|_{s,k}\approx h_k^{1-s}\|\bv\|_\LT+\|q\|_{H^s(\O)}\qquad\forall\,(\bv,q)\in\kES,
\end{equation*}
 where the constants in the norm equivalence depend on $s$.
\end{corollary}
%
\subsection{Another Scale of Mesh-Dependent Norms}\label{subsec:Another}
 In Section~\ref{sec:Analysis}.
 we will use the scale of mesh-dependent norms $\|\cdot\|_{s,k}$
 to analyze the effect of post-smoothing coupled with coarse grid correction.
   In order to analyze the effect of pre-smoothing
 coupled with coarse grid correction,
 we will need a second scale of mesh-dependent norms $\tbar\cdot\tbar_{s,k}$.
\par
 For $1\leq s\leq 2$, we define the mesh-dependent norm $\tbar\cdot\tbar_{s,k}$ by duality:
\begin{equation}\label{eq:SecondMDNorms}
  \tbar \vq\tbar_{s,k}=\sup_{\wr\in\FE{k}}
  \frac{\cB\big(\vq,\wr\big)}{\|\wr\|_{2-s,k}} \qquad\forall\,\vq\in \FE{k}.
\end{equation}
 It follows from \eqref{eq:StabilityEstimate}, \eqref{eq:OneEquivalence} and
 \eqref{eq:SecondMDNorms} that
\begin{equation}\label{eq:OneEquivalence2}
  \tbar\vq\tbar_{1,k}\approx \|\bv\|_\LT+\|q\|_\PHk \approx \|\vq\|_{1,k}
  \qquad\forall\,\vq\in \FE{k}.
\end{equation}
\par
 Note that the two scales of mesh-dependent norms together provide a generalized Cauchy-Schwarz inequality
 for the bilinear form $\cB(\cdot,\cdot)$:
\begin{equation}\label{eq:GCSInequality}
  \cB(\vq,\wr\big)\leq \tbar \vq\tbar_{1+\tau,k}\|\wr\|_{1-\tau,k}
\end{equation}
 for all $\vq,\wr\in\FE{k}$ and $0\leq\tau\leq 1$.
\section{Convergence Analysis}\label{sec:Analysis}
 In this section we will carry out the convergence analysis for the $W$-cycle algorithm, which
 is based on the smoothing and approximation properties \cite{BD:1981:WCycle,Hackbusch:1985:MMA}
 with respect to the mesh-dependent norms in Section~\ref{sec:MeshDependentNorms}.
 Once we have established these properties with respect to the scale of mesh-dependent norms defined
 in Section~\ref{subsec:MDNorms}, the analysis will proceed as in \cite[Section~5.3]{BLS:2014:StokesLame}.
\par
 Numerical results indicate that the $V$-cycle algorithm is also uniformly convergent in the
 nonconforming energy norm.
 But we will not consider the much more involved convergence analysis of the $V$-cycle algorithm in this paper.
\subsection{Smoothing and Approximation Properties}\label{subsec:Smoothing}
 Since the post-smoothing step in \eqref{eq:PostSmoothing} is just the Richardson
 relaxation for the SPD problem \eqref{eq:SPDProblem} and the operator
 $\BSB$ behaves like a typical SPD operator for second order problems
 (cf. \eqref{eq:SpectralRadius}),
 we have a standard smoothing property whose proof is identical to
 that of \cite[Lemma~5.1]{BLS:2014:StokesLame}.
\begin{lemma}\label{lem:SmoothingProperty}
 The estimate
\begin{equation}\label{eq:SmoothingProperty}
 \|R_k^m(\bv,q)\|_{1,k}\lesssim h_k^{-\tau}m^{-\tau/2}\|(\bv,q)\|_{1-\tau,k}\qquad\forall\,(\bv,q)\in\FE{k}
\end{equation}
 holds for $\tau\in[0,1]$.
\end{lemma}
\par
 The following approximation property is based on Corollary~\ref{cor:Comparison} and a duality argument.
\begin{lemma}\label{lem:ApproximationProperty}
 We have
\begin{equation*}
  \|(\CGC)\vq\|_{1-\alpha,k}\lesssim h_k^\alpha\|(\bv,q)\|_{1,k}
  \qquad\forall\,(\bv,q)\in \FE{k},
\end{equation*}
 where $\alpha\in(\frac12,1]$ is the index of elliptic regularity that appears
 in \eqref{eq:EllipticRegularityEst}.
\end{lemma}
\begin{proof} Let $\vq\in \FE{k}$ be arbitrary and
   $(\bzeta,\mu)=(\CGC)\vq$.
 In view of Corollary~\ref{cor:Comparison}, it suffices to show that
\begin{equation}\label{eq:Approximation1}
    h_k^{\alpha}\|\bzeta\|_{L_2(\O)}+\|\mu\|_{H^{1-\alpha}(\O)}\lesssim h^\alpha
   \|\vq\|_{1,k}.
\end{equation}
\par
 The estimate for $\bzeta$  follows immediately  from \eqref{eq:OneEquivalence} and
 Lemma~\ref{lem:Stability}:
\begin{equation}\label{eq:Approximation2}
  h_k^{\alpha}\|\bzeta\|_{L_2(\O)}\lesssim h_k^{\alpha}\|(\bzeta,\mu)\|_{1,k}\lesssim h^{\alpha}\|(\bv,q)\|_{1,k}.
\end{equation}
\par
 The estimate for $\mu$ is established through a duality argument.
 Let $\phi\in H^{-1+\alpha}(\O)$ and $(\bxi,\theta)\in [L_2(\O)]^d\times H^1_0(\O)$ satisfy
 \eqref{eq:ADarcy1}--\eqref{eq:ADarcy2} with $F$ replaced by $\phi$.
 Then we have
\begin{equation}\label{eq:Approximation4}
   \cB\big((\bxi,\theta),\wr\big)=\phi(r)\qquad\forall\,\wr\in \FE{k}
\end{equation}
 by Remark~\ref{rem:RHS}.  Moreover, if we define
 $(\bxi_{k-1},\theta_{k-1})\in \FE{k-1}$ by
\begin{equation}\label{eq:Approximation3}
  \cB\big((\bxi_{k-1},\theta_{k-1}),\wr\big)=\phi(r) \qquad\forall\,\wr\in\FE{k-1},
\end{equation}
 then
\begin{equation}\label{eq:Approximation5}
  \|\bxi-\bxi_{k-1}\|_\PLTk+\|\theta-\theta_{k-1}\|_\PHk\lesssim h_k^\alpha \|\phi\|_{H^{-1+\alpha}(\O)}
\end{equation}
 by the discretization error estimate \eqref{eq:ConcreteError}, since $h_{k-1}\approx h_k$.
\par
 It follows from \eqref{eq:NormEquivalence},
 \eqref{eq:BBdd}, \eqref{eq:GalerkinOrthogonality}, \eqref{eq:OneEquivalence} and
 \eqref{eq:Approximation3}--\eqref{eq:Approximation5} that
\begin{align*}
  \phi(\mu)&=\cB\big((\bxi,\theta),(\bzeta,\mu)\big)\\
      &=\cB\big((\bxi,\theta),(\CGC)(\bv,q)\big)\\
      &=\cB\big((\bxi,\theta)-(\bxi_{k-1},\theta_{k-1}),(\CGC)(\bv,q)\big)\\
      &=\cB\big((\bxi,\theta)-(\bxi_{k-1},\theta_{k-1}),(\bv,q)\big)\\
      &\lesssim (\|\bxi-\bxi_{k-1}\|_\PLTk+\|\theta-\theta_{k-1}\|_\PHk)
       (\|\bv\|_\PLTk+\|q\|_\PHk)\\
      &\lesssim h_k^\alpha \|\phi\|_{H^{-1+\alpha}(\O)}\|(\bv,q)\|_{1,k}
\end{align*}
 and hence, by duality,
\begin{equation}\label{eq:Approximation6}
  \|\mu\|_{H^{1-\alpha}(\O)}=\sup_{\phi\in H^{-1+\alpha}(\O)}
  \frac{\phi(\mu)}{\|\phi\|_{H^{-1+\alpha}(\O)}}\lesssim h_k^\alpha\|(\bv,q)\|_{1,k}.
\end{equation}
\par
 The estimate \eqref{eq:Approximation1} follows from \eqref{eq:Approximation2} and
 \eqref{eq:Approximation6}.
\end{proof}
\subsection{Convergence of the Two-Grid Algorithm}\label{subsec:TwoGrid}
 In the two-grid algorithm the coarse grid residual equation is
 solved exactly.  We can therefore set $E_{k-1}=0$ in \eqref{eq:MGRecursion} to
 obtain the error propagation of the two-grid algorithm, which is given by
 $R_k^{m_2}(\CGC)S_k^{m_1}$.
\par
 For $m\geq 1$,
 we have the following estimate on the
 effect of post-smoothing coupled with coarse grid correction
 by
 combining Lemma~\ref{lem:SmoothingProperty} and Lemma~\ref{lem:ApproximationProperty}.
\begin{equation} \label{eq:PostCoarse}
  \|R_k^{m}(\CGC)\vq\|_{1,k}\lesssim
     m^{-\alpha/2}\|\vq\|_{1,k}\qquad\forall\,\vq\in\FE{k}
\end{equation}
\par
 Using \eqref{eq:StabilityEstimate}, \eqref{eq:RitzProjection}, \eqref{eq:AdjointRelation},
 \eqref{eq:SecondMDNorms}, \eqref{eq:OneEquivalence2} and \eqref{eq:PostCoarse},
 we then obtain the following estimate on the
 effect of pre-smoothing coupled with coarse grid correction, where $m\geq1$.
\begin{equation}\label{eq:CoarsePre}
  \tbar (\CGC)S_k^m\vq\tbar_{1,k}\lesssim m^{-\alpha/2}\tbar\vq\tbar_{1,k}\qquad\forall\,\vq\in\FE{k}
\end{equation}
\par
 Therefore,
 for $m_1,m_2\geq 1$, we have
\begin{align}\label{eq:PostCoarsePre}
   &\|R_k^{m_2}(\CGC)S_k^{m_1}\vq\|_{1,k}\notag\\
   &\hspace{50pt}=\|R_k^{m_2}(\CGC)(\CGC)S_k^{m_1}\vq\|_{1,k}\\
   &\hspace{50pt}\lesssim (m_1m_2)^{-\alpha/2}\|\vq\|_{1,k}\hspace{100pt}\forall\,\vq\in\FE{k}\notag
\end{align}
 by \eqref{eq:Projection}, \eqref{eq:OneEquivalence2},
 \eqref{eq:PostCoarse} and \eqref{eq:CoarsePre}.
\begin{remark}\label{rem:Details}
  Since the estimates \eqref{eq:PostCoarse}--\eqref{eq:PostCoarsePre} are identical to the estimates
  (5.7)--(5.9) in \cite{BLS:2014:StokesLame}, we keep the arguments brief here
  and refer to \cite[Section~5]{BLS:2014:StokesLame} for the details.
\end{remark}
\par
 Putting \eqref{eq:PostCoarse}--\eqref{eq:PostCoarsePre} together, we arrive at the
 estimate
\begin{equation}\label{eq:TwoGridEst}
  \|R_k^{m_2}(\CGC)S_k^{m_1}\vq\|_{1,k}\leq C_* [\max(1,m_1)\max(m_2,1)]^{-\alpha/2}
    \|\vq\|_{1,k}
\end{equation}
 for all $\vq\in \FE{k}$ and $k\geq1$.  Thus the two-grid algorithm is a contraction if
 $\max(1,m_1)\max(m_2,1)$ is sufficiently large.
%
\subsection{Convergence of the $W$-Cycle Algorithm}\label{subsec:WCycle}
 The estimate \eqref{eq:TwoGridEst} and a perturbation argument lead to the following
 result for the $W$-cycle algorithm, whose proof is identical to that of
 \cite[Theorem~5.5]{BLS:2014:StokesLame}.
\begin{theorem}\label{thm:WCycle}
  Let $E_k$ be the error propagation operator for the $k$-th level $W$-cycle
  algorithm.
  For any $C_\dag>C_*$ $($the constant in \eqref{eq:TwoGridEst}$)$,
 there exists a positive number $m_*$ $($independent of $k)$ such that
\begin{equation*}
  \|E_k\vq\|_{1,k}\leq C_\dag\big(\max(1,m_1)\max(1,m_2)\big)^{-\alpha/2}
  \|\vq\|_{1,k}
\end{equation*}
 for all $\vq\in\FE{k}$ and $k\geq1$,
 provided $\max(1,m_1)\max(1,m_2)\geq m_*$.
\end{theorem}
\par
 Therefore, if $\max(1,m_1)\max(1,m_2)$ (independent of $k$)
 is sufficiently large, then the $W$-cycle algorithm is a contraction with respect to the nonconforming energy norm
 and the contraction number is bounded away from $1$ for $k\geq1$, i.e., the $W$-cycle algorithm
 converges uniformly.
\section{General Second Order Elliptic Problems}\label{sec:General}
 In this section we extend the multigrid results for the Darcy system to general
 second order elliptic problems of the form
\begin{equation}\label{eq:GeneralBVP}
  -\div\big(\bA\nabla p)+\bbeta\cdot\nabla p+\gamma p=f\quad\text{in}\;\O \quad \text{and} \quad
    p=g\quad\text{on}\;\p\O,
\end{equation}
 which include the Darcy system \eqref{eq:SDarcy1}--\eqref{eq:SDarcy3} as a special case,
 together with the adjoint problems
\begin{equation}\label{eq:AdjointBVP}
  -\div\big(\bA\nabla p)-\div(\bbeta p)+\gamma p=f\quad\text{in}\;\O \quad\text{and}
   \quad p=g \quad\text{on}\;\p\O.
\end{equation}
\par
 For the design and analysis of multigrid methods, it suffices to consider the case where $g=0$.
 We assume that $\bA$ is a (sufficiently) smooth SPD $d\times d$ matrix function on $\bar\O$,
 $\bbeta\in [W^1_\infty(\O)]^d$ and $\gamma\in L_\infty(\O)$.  We also assume that
 the boundary value problems \eqref{eq:GeneralBVP} and \eqref{eq:AdjointBVP} are both well-posed,
 which is the case if, for example,
\begin{equation}\label{eq:CoefficientCondition}
  \gamma-\frac12\div\bbeta\geq 0 \quad \text{a.e. in}\;\O.
\end{equation}
\subsection{Finite Element Methods}\label{subsec:NewFEM}
 The mixed finite element method for \eqref{eq:GeneralBVP} is to find
 $(\bu_h,p_h)\in\ESh$ such that
\goodbreak
\begin{alignat}{3}
  a(\bu_h,\bv)+b(\bv,p_h)&=0&\qquad&\forall\,\bv\in V_h,\label{eq:GRT1}\\
  b(\bu_h,q)-c_h(p_h,q)&=F(q)&\qquad&\forall\,q\in Q_h,\label{eq:GRT2}
\end{alignat}
 where the finite element space $\ESh$, the bilinear forms $a(\cdot,\cdot)$,
 $b(\cdot,\cdot)$ and the bounded linear functional $F$
 are identical to the ones for the Darcy system, and the mesh-dependent bilinear form
 $c_h(\cdot,\cdot)$ is defined by
\begin{equation}\label{eq:chDef}
  c_h(r,q)=\int_\O (\gamma r+\bbeta\cdot\nabla_h r)q\,dx \qquad\forall\,q,r\in Q_h.
\end{equation}
 Here $\nabla_h$ is the piecewise defined gradient operator.
\par
\par
 We will treat \eqref{eq:GRT1}--\eqref{eq:GRT2} as a nonconforming method for the
 following weak formulation of \eqref{eq:GeneralBVP}:  Find $(\bu,p)\in [\LT]^d\times H^1_0(\O)$
 such that
\begin{alignat}{3}
    a(\bu,\bv)+b'(\bv,p)&=0&\qquad&\forall\,\bv\in [\LT]^d,\label{eq:General1}\\
    b'(\bu,q)-c(p,q)&=F(q)&\qquad&\forall\,q\in H^1_0(\O),\label{eq:General2}
\end{alignat}
 where the bilinear form $b'(\cdot,\cdot)$ is identical to the one in \eqref{eq:ADarcy1}--\eqref{eq:ADarcy2}
 and
    $$c(r,q)=\int_\O(\gamma r+\bbeta\cdot\nabla r)q\,dx.$$
\par
 Similarly, the mixed finite element method for the adjoint problem \eqref{eq:AdjointBVP} is to find
 $(\bu_h,p_h)\in\ESh$ such that
\begin{alignat}{3}
  a(\bu_h,\bv)+b(\bv,p_h)&=0&\qquad&\forall\,\bv\in V_h,\label{eq:AdjointRT1}\\
  b(\bu_h,q)-c_h(q,p_h)&=F(q)&\qquad&\forall\,q\in Q_h,\label{eq:AdjointRT2}
\end{alignat}
 and it can be treated as a nonconforming method for the
 following weak formulation of \eqref{eq:AdjointBVP}:  Find $(\bu,p)\in [\LT]^d\times H^1_0(\O)$
 such that
\begin{alignat}{3}
    a(\bu,\bv)+b'(\bv,p)&=0&\qquad&\forall\,\bv\in [\LT]^d,\label{eq:Adjoint1}\\
    b'(\bu,q)-c(q,p)&=F(q)&\qquad&\forall\,q\in H^1_0(\O).\label{eq:Adjoint2}
\end{alignat}
\begin{remark}\label{rem:NewFEM}
  The discretizations \eqref{eq:GRT1}--\eqref{eq:GRT2} and
  \eqref{eq:AdjointRT1}--\eqref{eq:AdjointRT2} for the convection-diffusion-reaction problem
  \eqref{eq:GeneralBVP} and the advection-diffusion-reaction problem \eqref{eq:AdjointBVP} are
  different from the mixed finite
  element methods in \cite{DR:1982:Mixed} which are based on $\ES$ formulations.
  Instead, they are related to the upwind mixed finite element methods in \cite{Jaffre:1984:CD}.
\end{remark}
\begin{remark}\label{rem:SimilarProperties}
 Note that
 the systems \eqref{eq:General1}--\eqref{eq:General2} and
 \eqref{eq:Adjoint1}--\eqref{eq:Adjoint2} are well-posed for $F\in H^{-s}(\O)$
 ($0\leq s\leq 1$) and the elliptic regularity estimate \eqref{eq:EllipticRegularityEst} remains valid.
 Remark~\ref{rem:RHS} also holds for these problems.
\end{remark}
\subsection{Stability and Error Estimates}\label{subsec:GeneralSTabilityEstimate}
 Let $\cB_h(\cdot,\cdot)$ be the bilinear form on $\ESh$ defined by
\begin{equation}\label{eq:cBhDef}
  \cB_h\big(\vq,\wr\big)=a(\bv,\bw)+b(\bw,q)+b(\bv,r)-c_h(q,r).
\end{equation}
\begin{lemma}\label{lem:GeneralStabilityEst1}
  The stability estimate
\begin{align}\label{eq:GeneralStabilityEst1}
 &\|\bv\|_\PLT+\|q\|_\PH \\
    &\hspace{50pt}\approx\sup_{\wr\in\ESh}\frac{\cB_h(\vq,\wr\big)}{\|\bw\|_\PLT+\|r\|_\PH}
     \qquad \forall\,\vq\in\ESh\notag
\end{align}
 holds for sufficiently small $h$.
\end{lemma}
\begin{proof}  Let $S_h$ be the supremum on the right-hand side of \eqref{eq:GeneralStabilityEst1}.
  It suffices to show that
\begin{equation}\label{eq:GStability0}
  \|\bv\|_\LT+\|q\|_\PH\lesssim S_h \qquad\forall\,\vq\in \ESh,
\end{equation}
 since the opposite estimate follows from the results in Section~\ref{subsec:Stability} and
 the Poincar\'e-Friedrichs inequality \cite{Brenner:2003:PF}
\begin{equation}\label{eq:PF}
  \|q\|_\LT\lesssim \|q\|_\PH \qquad\forall\,q\in Q_h.
\end{equation}
\par
  For any $\vq\in \ESh$ we have, in view of \eqref{eq:chDef}, \eqref{eq:cBhDef} and \eqref{eq:PF},
  an obvious estimate
\begin{align}\label{eq:GStability1}
  a(\bv,\bv)&\leq \cB_h\big(\vq,(\bv,-q)\big)+C_1\|q\|_\PH\|q\|_\LT\\
    &\leq S_h\big(\|\bv\|_\LT+\|q\|_\PH\big)+C_1\|q\|_\PH\|q\|_\LT.\notag
\end{align}
\par
 Let $(\bzeta,\theta)\in \ESA$ satisfy
\begin{alignat*}{3}
    a(\bzeta,\bw)+b'(\bw,\theta)&=0&\qquad&\forall\,\bw\in [\LT]^d,\\
    b'(\bzeta,r)-c(r,\theta)&=\int_\O qr\,dx&\qquad&\forall\,r\in H^1_0(\O).
\end{alignat*}
 Then we have, by Remark~\ref{rem:SimilarProperties},
\begin{equation}\label{eq:GStability2}
  \cB_h\big(\wr,(\bzeta,\theta)\big)=\int_\O qr\,dx\qquad\forall\,\wr\in\ESA.
\end{equation}
 It follows from the elliptic regularity estimate \eqref{eq:EllipticRegularityEst} and the
 interpolation error estimates \eqref{eq:RTInterpolationEst1}, \eqref{eq:RTInterpolationEst2}
 and \eqref{eq:DGInterpolationEst}
 that
\begin{equation}\label{eq:GStability3}
  \|\bzeta-\Pi_h\bzeta\|_\PLT+\|\theta-\cI_h\theta\|_\PH\lesssim
     h^\alpha\|q\|_\LT,
\end{equation}
 which implies
\begin{equation}\label{eq:GStability4}
  \|\Pi_h\bzeta\|_\LT+\|\cI_h\theta\|_\PH\lesssim \|q\|_\LT.
\end{equation}
\par
 We have, by \eqref{eq:NormEquivalence}, \eqref{eq:BBdd}, \eqref{eq:GStability2}--\eqref{eq:GStability4},
\begin{align*}
  \|q\|_\LT^2&=\cB_h\big(\vq,(\bzeta,\theta)\big)\\
          &=\cB_h\big(\vq,((\bzeta-\Pi_h\bzeta),(\theta-\cI_h\theta))\big)+
              \cB_h\big(\vq,(\Pi_h\bzeta,\cI_h\theta)\big)\\
          &\leq C_2\big(\|\bv\|_\LT+\|q\|_\PH\big) h^\alpha\|q\|_\LT+
                C_3S_h\|q\|_\LT,
\end{align*}
 and hence
\begin{equation}\label{eq:GStability5}
  \|q\|_\LT\leq C_2h^\alpha \big(\|\bv\|_\LT+\|q\|_\PH\big)
       +C_3 S_h.
\end{equation}
 Combining \eqref{eq:GStability1} and \eqref{eq:GStability5}, we find
\begin{equation}\label{eq:GStability6}
  \|\bv\|_\LT^2\leq C_4 S_h\big(\|\bv\|_\LT+\|q\|_\PH\big)+
    C_5h^\alpha\|q\|_\PH\big(\|\bv\|_\LT+\|q\|_\PH\big).
\end{equation}
\par
 We also have, by \eqref{eq:NormEquivalence}, \eqref{eq:InfSupCondition} and
 \eqref{eq:cBhDef},
\begin{equation*}
  \|q\|_\PH\lesssim \sup_{\bw\in V_h}\frac{b(\bw,q)}{\|\bw\|_\LT}
  \lesssim\sup_{\bw\in V_h}\frac{\cB_h(\vq,(\bw,0)\big)}{\|\bw\|_\LT}+\|\bv\|_\LT\lesssim
   S_h+\|\bv\|_\LT,
\end{equation*}
 and hence
\begin{equation}\label{eq:GStability7}
  \|q\|_\PH^2\leq C_6\big(S_h^2+\|\bv\|_\LT^2\big).
\end{equation}
\par
 Putting \eqref{eq:GStability6} and \eqref{eq:GStability7} together, we arrive at
\begin{align}\label{eq:GStability8}
  \|\bv\|_\LT^2+\|q\|_\PH^2
  &\leq C_6S_h^2+(1+C_6)C_4 S_h\big(\|\bv\|_\LT+\|q\|_\PH\big)\\
    &\hspace{20pt}+ (1+C_6)C_5h^\alpha\|q\|_\PH\big(\|\bv\|_\LT+\|q\|_\PH\big).\notag
\end{align}
\par
 The estimate \eqref{eq:GStability0} follows from
 \eqref{eq:GStability8} and the inequality of arithmetic and geometric means
 provided $h$ is sufficiently small.
\end{proof}
\begin{remark}\label{rem:Schatz}
  The arguments in the proof of Lemma~\ref{lem:GeneralStabilityEst1} are motivated by
  the arguments of Schatz in \cite{Schatz:1974:Indefinite} for nonsymmetric and indefinite problems.
\end{remark}
\par
 Similar arguments yield the following stability result.
\begin{lemma}\label{lem:GeneralStabilityEst2}
 The stability estimate
\begin{align}\label{eq:GeneralStabilityEst2}
 &\|\bv\|_\PLT+\|q\|_\PH \\
    &\hspace{50pt}\approx\sup_{\wr\in\ESh}\frac{\cB_h(\wr,\vq\big)}{\|\bw\|_\PLT+\|r\|_\PH}
     \qquad \forall\,\vq\in\ESh\notag
\end{align}
 holds for sufficiently small $h$.
\end{lemma}
\par
 From now on we assume that \eqref{eq:GeneralStabilityEst1} and \eqref{eq:GeneralStabilityEst2}
 are valid for all the finite element spaces involved.
 It follows from these estimates and the  same arguments in Section~\ref{subsec:Error} that
 \eqref{eq:ConcreteError} and \eqref{eq:SmoothError} also hold for the solution
 $(\bu,p)$ of \eqref{eq:General1}--\eqref{eq:General2}
 (resp. \eqref{eq:Adjoint1}--\eqref{eq:Adjoint2}) and the solution
 $(\bu_h,p_h)$ of \eqref{eq:GRT1}--\eqref{eq:GRT2}
 (resp. \eqref{eq:AdjointRT1}--\eqref{eq:AdjointRT2}).
\subsection{Multigrid Algorithms}\label{subsec:GeneralMG}
 The set-up for the multigrid algorithms remains the same, but the definition of the operator
 $\Bk:\FE{k}\longrightarrow \FE{k}$ is modified as follows:
\begin{equation}\label{eq:GeneralBkDef}
  \big[\Bk\vq,\wr\big]_k=\cB_k\big(\vq,\wr\big)\qquad\forall\,\vq,\wr\in\FE{k},
\end{equation}
 where $\cB_k$ is the bilinear form on $\FE{k}$ defined by \eqref{eq:cBhDef}.
 The transpose $\Bk^t$ of $\Bk$ with respect to the mesh-dependent inner product
 $[\cdot,\cdot]_k$ satisfies
\begin{equation}\label{eq:BkTranspose}
  \big[\Bk^t\vq,\wr\big]_k=\cB_k\big(\wr,\vq\big)\qquad\forall\,\vq,\wr\in\FE{k}.
\end{equation}
\par
 We have the following analog of Lemma~\ref{lem:BkSkBk}, with an identical proof
 that uses \eqref{eq:GeneralStabilityEst1} instead of \eqref{eq:StabilityEstimate}.
\begin{lemma}\label{lem:BktSkBk}
  The norm equivalence
\begin{equation}\label{eq:BktSkBk}
  \big[\Bk^t\Sk\Bk\vq,\vq\big]_k^\frac12\approx \|\bv\|_\LT+\|q\|_\PHk \qquad\forall\,\vq\in\FE{k}
\end{equation}
 holds for $k=0,1,2,\ldots$.
\end{lemma}
\par
 Similar arguments using \eqref{eq:GeneralStabilityEst2} and \eqref{eq:BkTranspose}
 yield another analog of Lemma~\ref{lem:BkSkBk}.
\begin{lemma}\label{lem:BkSkBkt}
  The norm equivalence
\begin{equation}\label{eq:BkSkBkt}
  \big[\Bk\Sk\Bk^t\vq,\vq\big]_k^\frac12\approx \|\bv\|_\LT+\|q\|_\PHk \qquad\forall\,\vq\in\FE{k}
\end{equation}
 holds for $k=0,1,2,\ldots$.
\end{lemma}
\par
 In the definitions of the multigrid algorithms for the problem
\begin{equation}\label{eq:GeneralProblem1}
  \Bk\vq=(\bg,z)
\end{equation}
 arising from\eqref{eq:GRT1}--\eqref{eq:GRT2},
 the pre-smoothing step in \eqref{eq:PreSmoothing} becomes
\begin{equation}\label{eq:PreSmoothing1}
  (\bv_{j},q_{j})=(\bv_{j-1},q_{j-1})+\delta_k \Sk\Bk^t\big(\gz-\Bk(\bv_{j-1},q_{j-1})\big),
\end{equation}
 and the post-smoothing step in \eqref{eq:PostSmoothing} becomes
\begin{equation}\label{eq:PostSmoothing1}
  (\bv_{j},q_{j})=(\bv_{j-1},q_{j-1})+\delta_k \Bk^t\Sk\big(\gz-\Bk(\bv_{j-1},q_{j-1})\big)
\end{equation}
\par
 Similarly, in the definitions of the multigrid algorithms for the problem
\begin{equation}\label{eq:GeneralProblem2}
  \Bk^t\vq=(\bg,z)
\end{equation}
 arising from \eqref{eq:AdjointRT1}--\eqref{eq:AdjointRT2},
 the pre-smoothing step in \eqref{eq:PreSmoothing} becomes
\begin{equation}\label{eq:PreSmoothing2}
  (\bv_{j},q_{j})=(\bv_{j-1},q_{j-1})+\delta_k \Sk\Bk\big(\gz-\Bk^t(\bv_{j-1},q_{j-1})\big),
\end{equation}
 and the post-smoothing step in \eqref{eq:PostSmoothing} becomes
\begin{equation}\label{eq:PostSmoothing2}
  (\bv_{j},q_{j})=(\bv_{j-1},q_{j-1})+\delta_k \Bk\Sk\big(\gz-\Bk^t(\bv_{j-1},q_{j-1})\big).
\end{equation}
\par
 In view of \eqref{eq:BktSkBk} and \eqref{eq:BkSkBkt}, we can choose $\delta_k=Ch_k^2$ so that
\begin{equation}\label{eq:SpectralConditions}
  \delta_k\cdot\rho(\Bk^t\Sk\Bk)\leq 1 \quad \text{and} \quad
  \delta_k\cdot\rho(\Bk\Sk\Bk^t)\leq 1.
\end{equation}
\subsection{Convergence Analysis}\label{subsec:GeneralAnalysis}
 Since the approach is similar
 we will only point out the necessary modifications and refer to
 Section~\ref{sec:MeshDependentNorms} and Section~\ref{sec:Analysis} for details.
\par
 There are now four error propagation operators for the smoothing steps.
 The error propagation operator for one post-smoothing step
 is given by
\begin{equation}\label{eq:Rk1Def}
  R_k=Id_k-\delta_k \Bk^t\Sk\Bk,
\end{equation}
 in the case of \eqref{eq:PostSmoothing1}, and
\begin{equation}\label{eq:Rk2Def}
  \tR_k=Id_k-\delta_k \Bk\Sk\Bk^t,
\end{equation}
 in the case of \eqref{eq:PostSmoothing2}.
\par
 The error propagation operator for one pre-smoothing step  is given by
\begin{equation}\label{eq:Sk1Def}
  S_k=Id_k-\delta_k \Sk\Bk^t\Bk,
\end{equation}
 in the case of \eqref{eq:PreSmoothing1}, and
\begin{equation}\label{eq:Sk2Def}
  \tS_k=Id_k-\delta_k \Sk\Bk\Bk^t,
\end{equation}
 in the case of \eqref{eq:PreSmoothing2}.
\par
 These operators satisfy the following relations:
\begin{alignat}{3}
  \cB_k(R_k\vq,\wr\big)&=\cB_k\big(\vq,\tS_k\wr\big)&\qquad&\forall\,\vq,\wr\in\FE{k},
    \label{eq:AdjointRelation1}\\
  \cB_k\big(S_k\vq,\wr\big)&=\cB_k\big(\vq,\tR_k\wr\big)&\qquad&\forall\,\vq,\wr\in\FE{k}.
    \label{eq:AdjointRelation2}
\end{alignat}
\par
 For $0\leq s\leq 1$, there are two scales of mesh-dependent norms.
 The norm $\|\cdot\|_{s,k}$ is defined by
\begin{equation}\label{eq:FirstScale1}
  \|\vq\|_{s,k}=\big[(\Bk^t\Sk\Bk)^s\vq,\vq\big]_k^\frac12 \qquad\forall\,\vq\in \FE{k},
\end{equation}
 and the norm $\|\cdot\|_{s,k}^\sim$ is defined by
\begin{equation}\label{eq:FirstScale2}
  \|\vq\|_{s,k}^\sim=\big[(\Bk\Sk\Bk^t)^s\vq,\vq\big]_k^\frac12 \qquad\forall\,\vq\in \FE{k}.
\end{equation}
\par
 In view of \eqref{eq:VectorIP}--\eqref{eq:MDIP} and \eqref{eq:BktSkBk}--\eqref{eq:BkSkBkt},
 the norm equivalences
 \eqref{eq:ZeroEquivalence}--\eqref{eq:OneEquivalence} also hold for the norms
 $\|\cdot\|_{s,k}$ and $\|\cdot\|_{s,k}^\sim$ defined by
  \eqref{eq:FirstScale1}--\eqref{eq:FirstScale2}.  Consequently all the results
  in Section~\ref{subsec:MDNorms}--Section~\ref{subsec:DkChracterization}
   remain valid for these mesh-dependent norms, and in particular,
\begin{equation}\label{eq:FirstScalesEquivalent}
  \|\vq\|_{s,k}\approx \|\vq\|_{s,k}^\sim \qquad\forall\,\vq\in \FE{k}, \;k\geq1.
\end{equation}
  Moreover if we define, for
   $1\leq s\leq 2$, the norms
  $\tbar\cdot\tbar_{s,k}$ and $\tbar\cdot\tbar_{s,k}^\sim$ by
\begin{alignat}{3}
 \tbar \vq\tbar_{s,k}=\sup_{\wr\in\FE{k}}
  \frac{\cB_k\big(\vq,\wr\big)}{\|\wr\|_{2-s,k}} \qquad\forall\,\vq\in \FE{k},\label{eq:SecondMDNorms1}\\
 \tbar \vq\tbar_{s,k}^\sim=\sup_{\wr\in\FE{k}}
  \frac{\cB_k\big(\wr,\vq\big)}{\|\wr\|_{2-s,k}^\sim} \qquad\forall\,\vq\in \FE{k},\label{eq:SecondMDNorms2}
\end{alignat}
 then the results in Section~\ref{subsec:Another} also hold for these mesh-dependent norms.
\par
 There are now two Ritz projection operators.
 The operators $P_k^{k-1}:\FE{k}\longrightarrow \FE{k-1}$  and
 $\tP_k^{k-1}:\FE{k}\longrightarrow\FE{k-1}$ are defined by
\begin{alignat}{3}
  \cB_k\big(P_k^{k-1}\vq,\wr\big)=\cB_k(\vq,I_{k-1}^k\wr\big)
  &\qquad&\forall\,\vq,\wr\in\FE{k}, \label{eq:RitzProjection1}\\
  \cB_k\big(\vq,\tP_k^{k-1}\wr\big)=\cB_k\big(I_{k-1}^k\vq,\wr\big)&\qquad&\forall\,\vq,\wr\in\FE{k}.
  \label{eq:RitzProjection2}
\end{alignat}
 Property \eqref{eq:Projection} remains valid, and it also holds if $P_k^{k-1}$ is replaced by
 $\tP_k^{k-1}$.  Consequently we have the following analogs of the Galerkin orthogonality
 \eqref{eq:GalerkinOrthogonality}
\begin{align*}
  0&=\cB_k\big((Id_k-I_{k-1}^kP_k^{k-1})\vq,I_{k-1}^k\wr\big),\\
   0&=\cB_k\big(I_{k-1}^k \vq,(Id_k-I_{k-1}^k\tP_k^{k-1})\wr\big),
\end{align*}
 for all $\vq\in\FE{k}$ and $\wr\in\FE{k-1}$.
\par
 Note also that \eqref{eq:RitzProjection1} and \eqref{eq:RitzProjection2} imply
\begin{equation}\label{eq:RitzProjectionAdjointRelation}
  \cB_k\big(I_{k-1}^kP_k^{k-1}\vq,\wr\big)=\cB_k(\vq,I_{k-1}^k\tP_k^{k-1}\wr\big)
  \quad\forall\,\vq,\wr\in \FE{k}.
\end{equation}
\par
 The error propagation operators for the multigrid algorithms are given by \eqref{eq:MGRecursion}
 for the problem \eqref{eq:GeneralProblem1}, and
\begin{equation}\label{eq:MGRecursion2}
  \tE_k=\tR_k^{m_2}(Id_k-I_{k-1}^k\tP_k^{k-1}+I_{k-1}^k\tE_{k-1}^p\tP_k^{k-1})\tS_k^{m_1}
\end{equation}
 for the problem \eqref{eq:GeneralProblem2}.
\par
 Since the proofs of Lemma~\ref{lem:SmoothingProperty} and Lemma~\ref{lem:ApproximationProperty}
 only involve the results in Section~\ref{sec:MeshDependentNorms} and duality arguments
 based on elliptic regularity and Galerkin orthogonality,
 they remain valid for the norms $\|\cdot\|_{s,k}$ and $\|\cdot\|_{s,k}^\sim$ defined in
 \eqref{eq:FirstScale1} and \eqref{eq:FirstScale2}.  Therefore we have
 the estimates on the effect of post-smoothing coupled with coarse grid correction:
\begin{alignat}{3}
 \|R_k^{m}(\CGC)\vq\|_{1,k}&\lesssim
     m^{-\alpha/2}\|\vq\|_{1,k}&\qquad&\forall\,\vq\in\FE{k}.\label{eq:PostCoarse1}\\
 \|\tR_k^{m}(\CGCT)\vq\|_{1,k}^\sim&\lesssim
     m^{-\alpha/2}\|\vq\|_{1,k}^\sim&\qquad&\forall\,\vq\in\FE{k}.\label{eq:PostCoarse2}
\end{alignat}
\par
 It follows from \eqref{eq:GCSInequality},
 \eqref{eq:AdjointRelation2},
 \eqref{eq:FirstScalesEquivalent}, \eqref{eq:RitzProjectionAdjointRelation}
 and  \eqref{eq:PostCoarse2} that we have an estimate which
 measures the effect of pre-smoothing coupled with coarse grid correction for
 the problem \eqref{eq:GeneralProblem1}:
\begin{align}\label{eq:PreOnlyEstimate}
  &\tbar (\CGC)S_k^m\vq\tbar_{1,k}\notag\\
   &\hspace{30pt}=\sup_{\wr \in\FE{k}}
    \frac{\cB_k\big((\CGC)S_k^m\vq ,\wr \big)}{\|\wr \|_{1,k}}\notag\\
   &\hspace{30pt}=\sup_{\wr \in\FE{k}}
    \frac{\cB_k\big(\vq ,\tR_k^m(\CGCT)\wr \big)}{\|\wr \|_{1,k}}\\
    &\hspace{30pt}\leq\sup_{\wr \in\FE{k}}
    \frac{\tbar\vq\tbar_{1,k}\|\tR_k^m(\CGCT)\wr\|_{1,k}}{\|\wr \|_{1,k}}\notag\\
    &\hspace{30pt}\approx \sup_{\wr \in\FE{k}}
    \frac{\tbar\vq\tbar_{1,k}\|\tR_k^m(\CGCT)\wr\|_{1,k}^\sim}{\|\wr \|_{1,k}^\sim}\notag\\
    &\hspace{30pt}\lesssim m^{-\alpha/2}\tbar\vq\tbar_{1,k}\qquad\forall\,\vq\in\FE{k}.\notag
\end{align}
 Similarly the estimate
\begin{equation}\label{eq:CoarsePre2}
  \tbar (\CGCT)\tS_k^m\vq\tbar_{1,k}^\sim\lesssim m^{-\alpha/2}\tbar\vq\tbar_{1,k}^\sim
  \qquad\forall\,\vq\in\FE{k}
\end{equation}
 that measures the effect of pre-smoothing coupled with coarse grid correction for the
 problem \eqref{eq:GeneralProblem2}
 follows from \eqref{eq:GCSInequality},
 \eqref{eq:AdjointRelation1},
 \eqref{eq:FirstScalesEquivalent}, \eqref{eq:RitzProjectionAdjointRelation}
 and \eqref{eq:PostCoarse1}.
\par
 Consequently the estimate \eqref{eq:PostCoarsePre} for the two grid algorithm holds for the
 problem \eqref{eq:GeneralProblem1}, and its counterpart
\begin{align}\label{eq:TwoGridEst2}
  &\|\tR_k^{m_2}(\CGCT)\tS_k^{m_1}\vq\|_{1,k}^\sim\\
  &\hspace{40pt}\leq C_* [\max(1,m_1)\max(m_2,1)]^{-\alpha/2}
    \|\vq\|_{1,k}^\sim \qquad\forall\,\vq\in\FE{k},\;k\geq1\notag
\end{align}
 holds for the problem \eqref{eq:GeneralProblem2}.
\par
 A perturbation argument leads to the following convergence result for the $W$-cycle algorithm.
\begin{theorem}\label{thm:WCycle2}
  Let $E_k$ $($resp. $\tE_k)$ be the error propagation operator for the $k$-th level $W$-cycle
  algorithm for \eqref{eq:GeneralProblem1} $($resp. \eqref{eq:GeneralProblem2}$)$.
  For any $C_\dag>C_*$ $($the constant in \eqref{eq:TwoGridEst} and \eqref{eq:TwoGridEst2}$)$,
 there exists a positive number $m_*$ $($independent of $k)$ such that
\begin{align*}
  \|E_k\vq\|_{1,k}\leq C_\dag\big(\max(1,m_1)\max(1,m_2)\big)^{-\alpha/2}
  \|\vq\|_{1,k},\\
  \|\tE_k\vq\|_{1,k}^\sim\leq C_\dag\big(\max(1,m_1)\max(1,m_2)\big)^{-\alpha/2}
  \|\vq\|_{1,k}^\sim,
\end{align*}
 for all $\vq\in\FE{k}$ and $k\geq1$,
 provided $\max(1,m_1)\max(1,m_2)\geq m_*$.
\end{theorem}
%
\section{Numerical Results}\label{sec:Numerics}
 We report in this section numerical results
 that corroborate the theoretical estimates and illustrate the performance of the multigrid methods.
 The computational domains are the unit square $(0,1)^2$ and the $L$-shaped domain
 $(-1,1)^2\setminus[0,1]\times[-1,0]$.
 We use the Raviart-Thomas-N\'ed\'elec mixed finite element method of order $1$ on uniform meshes
 in all the numerical experiments, which were supported by the HPC resources of LONI.
\subsection{Error in the Nonconforming Energy Norm}\label{subsec:DiscretizationErrors}
 In this set of numerical experiments we solve the Darcy system
\begin{equation}\label{eq:DarcySystem}
  \bm{u}  = -\nabla p  \,\, \mbox{in} \,\, \Omega, \,\,
  \nabla \cdot \bm{u}  = f  \,\, \mbox{in} \,\, \Omega, \,\, \mbox{and} \,\,
   p  = 0 \,\, \mbox{on} \,\, \partial \Omega,
\end{equation}
 and the convection-diffusion equation
\begin{equation}\label{eq:CD}
 -\nabla \cdot \left( \nabla p \right) + \bm{\beta} \cdot \nabla p = f \,\, \mbox{in} \,\,
 \Omega \,\, \mbox{and} \,\, p
 = 0 \,\, \mbox{on} \,\, \partial \Omega,
\end{equation}
 where $\bm{\beta} = \left[ 2, -1 \right]^T$.
 We check the error estimate \eqref{eq:ConcreteError} by computing
\begin{equation*}
  \|\bm{u} - \bm{u}_h\|_\PLT  \quad \text{and} \quad
 \|p-p_h\|_\PH.
\end{equation*}
\subsubsection{Unit Square}\label{subsubsec:SquareRate}
 We take the exact solution to be
 $p = \sin\left(\pi x\right) \sin\left(\pi y\right)$ and $\bm{u} = -\nabla p$.
 The results are displayed in Table~\ref{table:DarcyRateSquare} and Table~\ref{table:CDRateSquare}.
 The index of elliptic regularity $\alpha=1$ for the square and the convergence rate for
 $\|p-p_h\|_\PH$ is $1$ for both problems, which agrees with \eqref{eq:ConcreteError}.  The convergence rate
 for $\|\bu-\bu_h\|_\PLT$ is $2$ for both problems, which is higher than the predicted rate of $1$.
 This is likely due to the phenomenon of superconvergence, since the exact solution is smooth and we use
 uniform meshes.
\begin{table}[hh]
\caption{Convergence rates for \eqref{eq:DarcySystem} on the unit square}
\centering
\begin{tabular}{c|c|c|c|c}
  \hline
  &&&&\\[-12pt]
 $h$ & $\|\bm{u} - \bm{u}_h\|_\PLT$ & rate & $\|p-p_h\|_\PH$  &
  rate \\
    &&&&\\[-12pt]\hline &&&&\\[-12pt]
$1/4$ & 4.213e-2  &  & 3.462e-1 &    \\
  \hline &&&&\\[-12pt]
$1/8$ & 1.026e-2  & 2.038 & 1.722e-1 & 1.008  \\
   \hline &&&&\\[-12pt]
$1/16$ & 2.529e-3  & 2.020 & 8.594e-2 & 1.003  \\
   \hline &&&&\\[-12pt]
$1/32$ & 6.281e-4  & 2.010 & 4.297e-2 & 1.000  \\
   \hline &&&&\\[-12pt]
$1/64$ & 1.565e-4  & 2.005 & 2.149e-2 & 1.000  \\
   \hline &&&&\\[-12pt]
$1/128$ & 3.905e-5  & 2.002 & 1.074e-2 & 1.000  \\ \hline
  \end{tabular}
\label{table:DarcyRateSquare}
\end{table}
\begin{table}[hh]
\caption{Convergence rates for \eqref{eq:CD} on the unit square}
\centering
\begin{tabular}{c|c|c|c|c}
\hline   &&&&\\[-12pt]
  $h$ & $\|\bm{u} - \bm{u}_h\|_\PLT$ & rate & $\|p-p_h\|_\PH$  & rate \\
  &&&&\\[-12pt]\hline
$1/4$ & 9.188e-2  &  & 3.492e-1 &    \\
  \hline&&&&\\[-12pt]
$1/8$ & 2.345e-2  & 1.970 & 1.727e-1 & 1.016  \\
  \hline&&&&\\[-12pt]
$1/16$ & 5.904e-3  & 1.990 & 8.600e-2 & 1.006  \\
  \hline&&&&\\[-12pt]
$1/32$ & 1.480e-3  & 1.996 & 4.298e-2 & 1.001  \\
  \hline&&&&\\[-12pt]
$1/64$ & 3.705e-4  & 1.998 & 2.150e-2 & 1.000  \\
  \hline&&&&\\[-12pt]
$1/128$ & 9.268e-5  & 1.999 & 1.075e-2 & 1.000  \\ \hline
  \end{tabular}
\label{table:CDRateSquare}
\end{table}
\subsubsection{$L$-Shaped Domain}\label{subsubsec:LShaprRate}
 We take the exact solution
  to be $p = (1-x^2)(1-y^2)r^{2/3} \sin\left(\left(2/3\right) \theta\right)$ and
  $\bm{u} =-\nabla p$, where $\left( r, \theta \right)$ are the polar coordinates.
  The index of elliptic regularity $\alpha$ can be any number $<\frac23$ for the $L$-shaped domain
  and the exact solution has the correct singularity.
  The results are presented in
  Table~\ref{table:DarcyLShapeRate} and Table~\ref{table:CDLShapeRate}, which agree with
  \eqref{eq:ConcreteError}.
\begin{table}[!ht]
\caption{Convergence rates for \eqref{eq:DarcySystem} on the $L$-shaped domain}
\centering
  \begin{tabular}{c|c|c|c|c}
\hline   &&&&\\[-12pt]
  $h$ & $\|\bm{u} - \bm{u}_h\|_\PLT$ & rate & $\|p-p_h\|_\PH$  & rate \\
 &&&&\\[-12pt] \hline    &&&&\\[-12pt]
$1/4$   & 1.121e-1 &       & 1.131e-1 &    \\
 \hline&&&&\\[-12pt]
$1/8$   & 6.654e-2 & 0.752 & 5.964e-2 & 0.923  \\
 \hline &&&&\\[-12pt]
$1/16$  & 4.144e-2 & 0.683 & 3.244e-2 & 0.879  \\
 \hline &&&&\\[-12pt]
$1/32$  & 2.605e-2 & 0.670 & 1.818e-2 & 0.835  \\
  \hline&&&&\\[-12pt]
$1/64$  & 1.640e-2 & 0.667 & 1.048e-2 & 0.795  \\
  \hline&&&&\\[-12pt]
$1/128$ & 1.033e-2 & 0.667 & 6.186e-3 & 0.760  \\ \hline
  \end{tabular}
\label{table:DarcyLShapeRate}
\end{table}
\begin{table}[!ht]
\caption{Convergence rates for \eqref{eq:CD} on the $L$-shaped domain}
\centering
  \begin{tabular}{c|c|c|c|c}
 \hline  &&&&\\[-12pt]
  $h$ & $\|\bm{u} - \bm{u}_h\|_\PLT$ & rate & $\|p-p_h\|_\PH$  & rate \\
   &&&&\\[-12pt] \hline &&&&\\[-12pt]
$1/4$   & 1.306e-1  &       & 1.156e-1 &   \\
  \hline&&&&\\[-12pt]
$1/8$   & 7.025e-2  & 0.894 & 6.023e-2 & 0.941  \\
  \hline&&&&\\[-12pt]
$1/16$  & 4.231e-2  & 0.731 & 3.260e-2 & 0.886  \\
  \hline&&&&\\[-12pt]
$1/32$  & 2.629e-2  & 0.687 & 1.823e-2 & 0.838  \\
  \hline&&&&\\[-12pt]
$1/64$  & 1.647e-2  & 0.674 & 1.049e-2 & 0.797  \\
 \hline&&&&\\[-12pt]
$1/128$ & 1.035e-2  & 0.670 & 6.191e-3 & 0.761  \\ \hline
  \end{tabular}
\label{table:CDLShapeRate}
\end{table}
\goodbreak
\subsection{Convergence of Multigrid Methods}\label{subsec:MGNumerics}
 In this set of experiments we carry out the symmetric $W$-cycle and $V$-cycle algorithms with
 $m$ pre-smoothing and $m$ post-smoothing steps
 for the Darcy system \eqref{eq:DarcySystem} and the convection-diffusion equation \eqref{eq:CD}.
 We use the multigrid $V(4,4)$
 algorithm for interior penalty methods to generate the preconditioner $L_k$ in \eqref{eq:LkDef}--\eqref{eq:bSkDef}.
 We report the contraction numbers obtained by computing the largest eigenvalue of
 the error propagation operators.  The mesh size at level $k$ is $2^{-k}$.
\subsubsection{Unit Square}\label{subsubsec:MGSquare}
 The contraction numbers of the $W$-cycle algorithms for \eqref{eq:DarcySystem} and
 \eqref{eq:CD} for various $m$ and $k=1,\ldots,6$
 are presented in Table~\ref{table:WCycleDarcySquare} and
 Table~\ref{table:WCycleCDSquare}.
 For both problems the asymptotic decay rate of $1/m$ for the contraction number
 predicted by Theorem~\ref{thm:WCycle} and Theorem~\ref{thm:WCycle2} is observed.
 (The index of elliptic regularity $\alpha$ for the unit square is $1$.)
 This is also confirmed by the log-log graph in Figure~\ref{fig:WCycleDarcySquare}, where
 the contraction number of the $W$-cycle algorithm for \eqref{eq:DarcySystem} is
 plotted against the number of smooth steps $m$.
\begin{table}[!ht]
\caption{Contraction numbers of the $W$-cycle multigrid method for \eqref{eq:DarcySystem} on the unit square}
\centering
  \begin{tabular}{c|c|c|c|c|c|c}
    \hline
 & $k=1$  & $k=2$  & $k=3$  & $k=4$  & $k=5$  & $k=6$ \\ \hline
$m=10$ & 0.80  & 0.81  & 0.81  & 0.81  & 0.81  & 0.81 \\ \hline
$m=20$ & 0.66  & 0.67  & 0.67  & 0.67  & 0.67  & 0.67 \\ \hline
$m=40$ & 0.47  & 0.48  & 0.48  & 0.48  & 0.48  & 0.48 \\ \hline
$m=80$ & 0.24  & 0.24  & 0.24  & 0.24  & 0.24  & 0.24 \\ \hline
  \end{tabular}
\label{table:WCycleDarcySquare}
\end{table}
\begin{table}[!ht]
\caption{Contraction numbers of the $W$-cycle multigrid method for \eqref{eq:CD} on the unit square}
\centering
  \begin{tabular}{c|c|c|c|c|c|c}
    \hline
 & $k=1$  & $k=2$  & $k=3$  & $k=4$  & $k=5$  & $k=6$ \\ \hline
$m=10$ & 0.80  & 0.81  & 0.81  & 0.81  & 0.81  & 0.81 \\ \hline
$m=20$ & 0.67  & 0.68  & 0.67  & 0.67  & 0.67  & 0.67 \\ \hline
$m=40$ & 0.48  & 0.48  & 0.48  & 0.48  & 0.48  & 0.48 \\ \hline
$m=80$ & 0.24  & 0.24  & 0.24  & 0.24  & 0.24  & 0.25 \\ \hline
  \end{tabular}
\label{table:WCycleCDSquare}
\end{table}
\begin{figure}[!ht]
  \centering
 \includegraphics[scale=0.4]{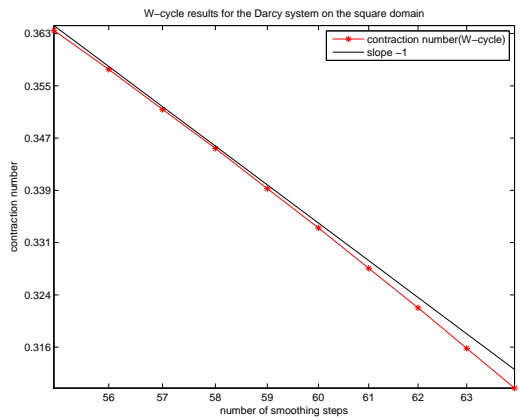}
  \caption{Contraction numbers of the $W$-cycle multigrid method for \eqref{eq:DarcySystem} on the unit square}
  \label{fig:WCycleDarcySquare}
\end{figure}
\par
 We also report the contraction numbers of the $V$-cycle algorithm for \eqref{eq:DarcySystem} and \eqref{eq:CD}
 in Table~\ref{table:VCycleDarcySquare} and Table~\ref{table:VCycleCDSquare}.
 They are similar to the contraction numbers for the $W$-cycle algorithm in
 Table~\ref{table:WCycleDarcySquare} and
 Table~\ref{table:WCycleCDSquare} but are slightly larger.
\begin{table}[!hh]
\caption{Contraction numbers of the $V$-cycle multigrid method for \eqref{eq:DarcySystem} on the unit square}
\centering
  \begin{tabular}{c|c|c|c|c|c|c}
    \hline
 & $k=1$  & $k=2$  & $k=3$  & $k=4$  & $k=5$  & $k=6$ \\ \hline
$m=10$ & 0.80  & 0.82  & 0.81  & 0.82  & 0.82  & 0.82 \\ \hline
$m=20$ & 0.66  & 0.68  & 0.68  & 0.68  & 0.68  & 0.68 \\ \hline
$m=40$ & 0.47  & 0.48  & 0.48  & 0.48  & 0.48  & 0.48 \\ \hline
$m=80$ & 0.24  & 0.25  & 0.25  & 0.25  & 0.25  & 0.25 \\ \hline
  \end{tabular}
\label{table:VCycleDarcySquare}
\end{table}
\begin{table}[!hh]
\caption{Contraction numbers of the $V$-cycle multigrid method for \eqref{eq:CD} on the unit square}
\centering
  \begin{tabular}{c|c|c|c|c|c|c}
    \hline
 & $k=1$  & $k=2$  & $k=3$  & $k=4$  & $k=5$  & $k=6$ \\ \hline
$m=10$ & 0.80  & 0.82  & 0.82  & 0.82  & 0.82  & 0.82 \\ \hline
$m=20$ & 0.67  & 0.68  & 0.68  & 0.68  & 0.68  & 0.68 \\ \hline
$m=40$ & 0.48  & 0.49  & 0.49  & 0.49  & 0.49  & 0.49 \\ \hline
$m=80$ & 0.24  & 0.25  & 0.25  & 0.25  & 0.25  & 0.25 \\ \hline
  \end{tabular}
\label{table:VCycleCDSquare}
\end{table}
\subsubsection{$L$-Shaped Domain}\label{subsubsec:MGLShape}
 The contraction numbers of the $W$-cycle algorithms for \eqref{eq:DarcySystem} and
 \eqref{eq:CD} for various $m$ and $k=1,\ldots,6$
 are displayed in Table~\ref{table:WCycleDarcyLShape} and
 Table~\ref{table:WCycleCDLShape}.  The contraction numbers are larger than the corresponding
 contraction numbers for the unit square, which is consistent with
 Theorem~\ref{thm:WCycle} and Theorem~\ref{thm:WCycle2} since the index of elliptic regularity $\alpha$
 for the $L$-shaped domain is less than $\frac23$.  This is also confirmed by the log-log graph in
 Figure~\ref{fig:WCycleDarcyLShape}, where the contraction number of the $W$-cycle algorithm for
 \eqref{eq:DarcySystem} is plotted against the number of smoothing steps $m$.
\par
 The contraction numbers for the $V$-cycle algorithm are similar and therefore not reported.
\begin{table}[!hh]
\caption{Contraction numbers of the $W$-cycle multigrid method for \eqref{eq:DarcySystem} on the $L$-shaped
 domain}
\centering
  \begin{tabular}{c|c|c|c|c|c|c}
    \hline
 & $k=1$  & $k=2$  & $k=3$  & $k=4$  & $k=5$  & $k=6$ \\ \hline
$m=10$ & 0.81  & 0.82  & 0.82  & 0.82  & 0.82  & 0.82 \\ \hline
$m=20$ & 0.70  & 0.70  & 0.70  & 0.70  & 0.70  & 0.70 \\ \hline
$m=40$ & 0.51  & 0.51  & 0.51  & 0.51  & 0.51  & 0.51 \\ \hline
$m=80$ & 0.28  & 0.28  & 0.28  & 0.28  & 0.28  & 0.28\\ \hline
  \end{tabular}
\label{table:WCycleDarcyLShape}
\end{table}
\begin{table}[!hh]
\caption{Contraction numbers of the $W$-cycle multigrid method for \eqref{eq:CD} on the $L$-shaped
 domain}
\centering
  \begin{tabular}{c|c|c|c|c|c|c}
    \hline
 & $k=1$  & $k=2$  & $k=3$  & $k=4$  & $k=5$  & $k=6$ \\ \hline
$m=10$ & 0.81  & 0.82  & 0.82  & 0.82  & 0.82  & 0.82 \\ \hline
$m=20$ & 0.70  & 0.70  & 0.70  & 0.70  & 0.70  & 0.70 \\ \hline
$m=40$ & 0.52  & 0.52  & 0.52  & 0.52  & 0.52  & 0.52 \\ \hline
$m=80$ & 0.29  & 0.29  & 0.29  & 0.29  & 0.29  & 0.29\\ \hline
  \end{tabular}
\label{table:WCycleCDLShape}
\end{table}
\begin{figure}[!hh]
  \centering
 \includegraphics[scale=0.4]{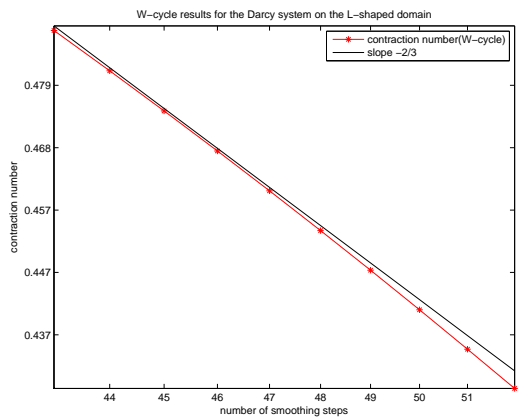}
  \caption{Contraction numbers of the $W$-cycle multigrid method for \eqref{eq:DarcySystem} on the $L$-shaped domain}
  \label{fig:WCycleDarcyLShape}
\end{figure}
%
\section{Concluding Remarks}\label{sec:Conclusion}
 In this paper we developed multigrid algorithms for the Darcy system discretized by
  Raviart-Thomas-N\'ed\'elec mixed finite element methods
 of order at least 1, and showed that with minimal modifications the multigrid algorithms can
 also be applied to convection-diffusion-reaction and advection-diffusion-reaction problems.
 Note that the number of degrees of freedom of the Raviart-Thomas-N\'ed\'elec mixed finite element method
 of order 1 associated with a triangulation $\cT_h$ is less than the number of degrees of freedom of
 the Raviart-Thomas-N\'ed\'elec mixed finite element method of order 0 associated with the triangulation
 $\cT_{h/2}$ obtained from $\cT_h$ by uniform refinement.  Therefore, from the point of view of multigrid,
 the requirement that the order of the method has to be at least 1 is not restrictive.
\par
 The results in this paper can be extended to rectangular Raviart-Thomas-N\'ed\'elec mixed finite element methods
  and to other stable mixed finite element methods for the Darcy system.  It should also be possible to extend
  our approach to mixed finite element methods for linear elasticity that are based on a stress-displacement formulation
  (cf. \cite{BBF:2013:Mixed} and the references therein).  We note that a nonstandard analysis similar to
  the one in Section~\ref{sec:Analysis} has been carried out in \cite{Stenberg:1988:Elasticity} for a family of
  such finite element methods.
\par
  Finally it would be interesting to extend our approach to the upwind mixed finite element methods
  for convection dominated problems  in \cite{Jaffre:1984:CD} (cf. also \cite{KP:2008:CD}).
%
\par\smallskip\noindent
{\bf Acknowledgement} \quad The authors would like to thank Monique Dauge for helpful discussions
 on the regularity of the Darcy system.

\end{document}